\documentclass[10pt]{amsart}
\usepackage{amssymb,verbatim,hyperref,enumerate}
\newcounter{desccount}
\newcommand{\descitem}[1]{%
	  \item[#1] \refstepcounter{desccount}\label{#1}
}
\newcommand{\descref}[1]{\hyperref[#1]{#1}}

\newcommand{\vmoellipticity}{\hyperref[vmoellipticity]{\textbf{(V1)}}}
\newcommand{\vmogrowth}{\hyperref[vmogrowth]{\textbf{(V3)}}}
\newcommand{\vmodifferentiability}{\hyperref[vmodifferentiability]{\textbf{(V2)}}}
\newcommand{\vmoctsdifferentiability}{\hyperref[vmoctsdifferentiability]{\textbf{(V5)}}}
\newcommand{\vmocontinu}{\hyperref[vmocontinu]{\textbf{(V4)}}}
\newcommand{\vmo}{\hyperref[vmo]{\textbf{(V6)}}}

\newcommand{\vmocontrol}{\hyperref[vmocontrol]{\textbf{(B)}}}

\newcommand{\R}{\mathbb{R}}
\newcommand{\N}{\mathbb{N}}
\newcommand{\Rn}{\mathbb{R}^n}
\newcommand{\Rm}{\mathbb{R}^m}
\newcommand{\RN}{\mathbb{R}^N}
\newcommand{\RnN}{\mathbb{R}^{nN}}

\newcommand{\A}{\mathcal{A}}
\newcommand{\xo}{x_{0}}

\newcommand{\lx}{\ell(\xo)}

\newcommand{\lrhoxo}{\ell_{\xo,\rho}}
\newcommand{\lthetarhoxo}{\ell_{\xo,\theta\rho}}
\newcommand{\lthetakrhoxo}{\ell_{\xo,\theta^k\rho}}

\newcommand{\lplusthetakrhoxo}{\ell_{\xo,\theta^k\rho}}
\newcommand{\lplusthetakonerhoxo}{\ell_{\xo,\theta^{k+1}\rho}}
\newcommand{\lplushalfrhoxo}{\ell_{\xo,\frac{\rho}{2}}}
\newcommand{\lhalfrhoxo}{\ell_{\xo,\frac{\rho}{2}}}

\newcommand{\lplusthetarhoxo}{\ell_{\xo,\theta\rho}}
\newcommand{\lplusrhoxo}{\ell_{\xo,\rho}}

\newcommand{\rhoo}{\rho_{0}}

\newcommand{\brhoxo}{B_{\rho}(x_{0})}

\newcommand{\bhalfrhoxo}{B_{\frac{\rho}{2}}(x_{0})}

\newcommand{\brxo}{B_{r}(\xo)}

\newcommand{\dnu}{\left(Du\right)_{\xo,\rho}}
\newcommand{\dtwo}{\left(Du\right)_{\xo,2\rho}}

\newcommand{\dnuk}{\left(Du\right)_{\xo,\theta^k\rho}}

\newcommand{\urho}{\left(u\right)_{\xo,\rho}}
\newcommand{\uhalfrho}{(u)_{\bhalfrhoxo}}

\newcommand{\uplushalfrho}{(u)_{\bhalfrhoxo}}

\newcommand{\btworhoxo}{B_{2\rho}(\xo)}

\newcommand{\bthetarhoxo}{B_{\theta\rho}(\xo)}

\newcommand{\bthetak}{B_{\theta^k\rho}(x_{0})}
\newcommand{\bthetakone}{B_{\theta^{k+1}\rho}(x_{0})}

\newcommand{\wpx}{W^{1,p(\cdot)}(\Omega,\RN)}

\newcommand{\wpxo}{W^{1,p(\cdot)}_0(\Omega,\RN)}

\newcommand{\nl}{\\&\qquad\qquad}
\newcommand{\dx}{\,dx}
\newcommand{\dy}{\,dy}
\newcommand{\dt}{\,dt}

\newcommand{\vprop}{Lemma~\ref{vproperties}~}

\newcommand{\mathand}{\qquad\mbox{and}\qquad}

 \newcommand{\fint}{-\hskip-1.08em\int}

\DeclareMathOperator{\Hom}{Hom}

\DeclareMathOperator{\sing}{Sing}
\DeclareMathOperator{\reg}{Reg}
\DeclareMathOperator{\diverge}{div}
\DeclareMathOperator{\dimh}{dim_{\mathcal{H}}}

\theoremstyle{plain}
  \newtheorem{theorem}[subsection]{Theorem}

  \newtheorem{lemma}[subsection]{Lemma}
  \newtheorem{corollary}[subsection]{Corollary}

\theoremstyle{remark}
  \newtheorem{remark}[subsection]{Remark}

\theoremstyle{definition}

\begin{document}

\title[Elliptic systems with discontinuous coefficients and $p(x)$-growth]{Partial continuity for nonlinear systems with nonstandard growth and discontinuous coefficients}
\author{Chris van der Heide}
\address{School of Mathematics and Physics\\ University of Queensland\\ Brisbane, Qld 4072\\ Australia}
\email{ c.vanderheide@uq.edu.au}
\subjclass[2010]{35B65, 35J47}
\thanks{This research was supported by an Australian Postgraduate Award and the Australian Research Council grant DP120101886}

\vspace{-0.3in}
\begin{abstract}
We obtain new partial H\"older continuity results for solutions to divergence form elliptic systems with discontinuous coefficients, obeying $p(x)$-type nonstandard growth conditions. By an application of the method of $\mathcal{A}$-harmonic approximation, we are able to allow for both VMO-discontinuities in the coefficients, and the minimal $\log$-H\"older regularity assumption on the exponent function. In doing so, we recover a local version of the quantisation phenomenon characteristic continuous coefficient case.
\end{abstract}

\maketitle

\section{Introduction}

In the past 30 years, there has been substantial interest in the study of nonstandard growth problems in elliptic PDE. Concurrently, an increasingly rich theory has been developed concerning equations and systems with irregular coefficients. In this paper we consider the partial H\"older continuity of solutions to systems of nonlinear elliptic PDE in divergence form, where the coefficients may have discontinuities with vanishing mean oscillation. In particular, we are concerned with weak solutions to the problem 
\begin{align}\label{vmoweaka}
	- \diverge\, a(x,u,Du)  &= b(x,u,Du) \quad\text{in }\ \Omega,
\end{align}
for some bounded open set $\Omega\subset\Rn$, where $n\geq2$. A weak solution is interpreted as a function $u\in W^{1,p(\cdot)}(\Omega,\RN)$, $N\geq2$, such that for any fixed $\phi\in \wpxo$
\begin{align*}
	\int_{\Omega}a(x,u,Du)D\phi\dx =\int_{\Omega}b(x,u,Du)\phi\dx.
\end{align*}
Here, $W^{1,p(\cdot)}(\Omega,\RN)$, is the space $\{u\in L^{1,1}(\Omega,\RN) : |Du|^{p(x)}\in L^1(\Omega)\}$, and $\wpxo$ is the closure of smooth functions with compact support with respect to the $W^{1,p(\cdot)}(\Omega,\RN)$ norm. We assume the usual Ahlfor's condition on $\Omega$, that there exists some $k>0$ such that $|\brhoxo\cap\Omega|\geq k\rho^n$ for every $\xo\in\Omega$ and $0<\rho\leq\text{diam}(\Omega)$, where $|S|$ is the Lebesgue measure of the set $S$, and $\brhoxo = \{x\in\Rn:|x-\xo|<\rho\}$.

The well-posedness of problems with VMO coefficients dates back to the work of Chiarenza et al., who developed a priori estimates from \cite{chiarenza_1991} into an existence theory in \cite{chiarenza_1993}, see also \cite{bramanti_cerutti}. The theory was further developed by Byun in \cite{byun_2005,byun_2005a}, see also the references within. The work of Krylov highlights the fact that problems of this type naturally arise when studying stochastic processes \cite{khasminskii_krylov, krylov_2002, krylov_2004}. This subsequently lead to the study of both elliptic and parabolic equations in \cite{krylov_2007a, krylov_2007}, as well as problems in control theory \cite{krylov_2010} and stochastic PDE \cite{krylov_2009}. Fully nonlinear equations have been studied by Dong, Krylov, and Li in \cite{dong_krylov_li}, but higher dimensional analogues are less well characterised. Ragusa and Tachikawa have also obtained numerous results for certain functionals in the variational setting \cite{ragusa_tachikawa_2005,ragusa_tachikawa_2005a,ragusa_tachikawa_2008,ragusa_tachikawa_2011}, including $p(x)$-harmonic maps with Takabayashi in \cite{ragusa_tachikawa_takabayashi_2013}, see \cite{ragusa_tachikawa_2014} for review.

In 2008 Foss and Mingione \cite{foss_mingione_2008} settled the longstanding open problem of partial H\"older continuity for continuous coefficient systems. They adapted harmonic approximation techniques, which had previously been used to obtain gradient continuity results. Until this landmark paper, results in this setting had required some special structure on the PDE or restrictions on the dimension (see \cite{foss_mingione_2008} for discussion). While the gradient of solutions to systems with H\"older continuous coefficients in some sense inherit the system's H\"older continuity, the H\"older exponents degenerate together as well \cite{duzaar_grotowski_2000, beck_2007, hamburger_2007}, (see also \cite{duzaar_mingione_2009, mingione_dark}). Dini continuity represents the borderline case \cite{duzaar_gastel_2002}, where some control on the gradient's modulus of continuity is retained.

Note that the technique developed in \cite{foss_mingione_2008} only asks for H\"older continuity of the solution, allowing for gradient blowup at regular points. Consequently, they prove a type of `$\varepsilon$-regularity' or `quantisation of singularities' result, whereby in order for a point to be singular (with respect to a fixed H\"older exponent), a certain local energy needs to exceed a finite quantity at all scales. The subquadratic analog can be retrieved from \cite{habermann_2013}, where Habermann treated it as a special case of the nonstandard growth problem. However, due to technical obstructions, the method no longer allows for gradient blowup, leading to localised quantisation effects that degenerate as the gradient explodes.

It was then observed in \cite{bogelein_duzaar_habermann_scheven_2011} that the technique developed in \cite{foss_mingione_2008} allows the continuity assumption on the coefficients to be relaxed in a controlled way, admitting a class of systems that may be discontinuous in the spatial variable. In this work, the authors demonstrated that under a VMO assumption in the superquadratic growth case, the gradient belongs to a certain Morrey space. This in turn implies H\"older continuity of the solution $u$ for every exponent $\alpha\in(0,1)$. This technique has also been carried over to problems with nonstandard growth in the context of stationary electrorheological fluids \cite{bogelein_duzaar_habermann_scheven_2012}.

In the current work we recover a local analogue of Foss and Mingione's $\varepsilon$-regularity result in the VMO setting. This improves upon the previous results in a number of ways, covering both the VMO and variable exponent cases in one step, while allowing for a relaxation of standard growth assumptions on the exponent $p(x)$.

We assume a natural energy bound on the solution, and the so-called $\log$-H\"older continuity condition on the exponent function $p:\Omega\to[\gamma_1,\gamma_2]$, for $1<\gamma_1\leq\gamma_2<\infty$. That is, for some $E,L>0$ there holds
\begin{align}\label{vmomodfinite}
	\int_{\Omega}|Du|^{p(x)} \leq E <\infty\qquad\text{and}\qquad
	\limsup_{\rho\downarrow0}\omega_p(\rho)\log\left(\frac{1}{\rho}\right)\leq L.
\end{align}
These are well known to be necessary for both the well-posedness of variable exponent nonstandard growth problems, and in order to obtain basic higher integrability results \cite{zhikov_1995, zhikov_1997}. However, continuity and higher regularity results usually require at least that $L=0$ (see, for example, \cite{bogelein_duzaar_habermann_scheven_2012, habermann_2013}). To the authors knowledge, together with the other results proved in \cite{vdh_phd}, these are the first continuity results obtained in the $p(x)$-growth setting while assuming only \eqref{vmomodfinite}.

Write $\Hom(\Rn\otimes\RN)$ for the space of pointwise linear maps from $\Rn$ to $\RN$. We assume the Carath\'edory vector field $a:\Omega\times\RN\times\Hom(\Rn\otimes\RN)$ is Borel measurable, with continuously differentiable partial map $z\mapsto a(\cdot,\cdot,z)$. Furthermore, we have the standard ellipticity and nonstandard growth conditions, that for some $1<\gamma_1\leq p(x)\leq\gamma_2<\infty$ and $0<\nu\leq 1\leq L<\infty$, there holds
\begin{description}
	\descitem{(V1)}\label{vmoellipticity}
			$\nu (1 + |z|)^{p(x)-2}|\zeta|^2 \leq D_za(x,\xi,z)\zeta\cdot \zeta,$
	\descitem{(V2)}\label{vmodifferentiability}
			$|D_{z}a(x,\xi,z)| \leq L(1 + |z|)^{p(x)-2},$
	\descitem{(V3)}\label{vmogrowth}
			$|a(x,\xi,z)| \leq L(1 + |z|)^{p(x)-1},$
\end{description}
for all $(x,\xi,z)\in\Omega\times\RN\times\Hom(\Rn\otimes\RN)$ and $\zeta\in \Hom(\Rn\otimes\RN)$.

We further assume continuity in the second variable, with bounded, concave, non-decreasing modulus of continuity $\omega$, and that $D_z a$ has a modulus of continuity $\mu$. That is, $\omega_{\xi},\mu:[0,\infty)\to[0,1]$ satisfy $\lim_{r\downarrow0}\omega_{\xi}(r)=0$ and $\lim_{r\downarrow0}\mu(r)=0$, and
\begin{description}
	\descitem{(V4)}\label{vmocontinu}
			$\big|a(x,\xi,z) - a(x,\hat{\xi},z)\big| \leq L\omega_{\xi}\big(|\xi - \hat{\xi}|\big)(1 + |z|)^{p(x)-1},$
	\descitem{(V5)}\label{vmoctsdifferentiable}
			$\big|D_za(x,\xi,z) - D_za(x,\xi,\bar{z})\big|
			\leq \begin{cases} L\mu\bigg(\frac{|z-\bar{z}|}{1 + |z| + |\bar{z}|}\bigg)\big(1 + |z| + |\bar{z}|\big)^{p(x)-2} & 2 \leq p(x) \\ 
L\mu\bigg(\frac{|z-\bar{z}|}{1 + |z| + |\bar{z}|}\bigg)\bigg(\frac{1 + |z| + |\bar{z}|}{(1 + |z|)(1 + |\bar{z}|)}\bigg)^{2-p(x)} &1< p(x)< 2. \end{cases}$
\end{description}
We make no assumptions regarding the continuity of the vector field $a$ in its first variable, requiring only the VMO-type condition on the map $x\mapsto\frac{a(x,\xi,z)}{(1 + |z|^2)^{p-1}}$. That is for $\xo\in\Omega$, $r\in(0,\rhoo]$, $\xi\in\RN$ and $z\in\RnN$ there holds
\begin{description}
	\descitem{(V6)}\label{vmo}
		$\big|a(x,\xi,z) - (a(x,\xi,z))_{\xo,r}\big| \leq \mathbf{v}_{\xo}(x,r)\Big[(1 + |z|)^{p(x)-1} + (1 + |z|)^{p(\xo)-1}\Big]\Big[1 + \log\left(1 + |z|\right)\Big]$
\end{description}
uniformly in $\xi$ and $z$. Here, $\mathbf{v}_{\xo}:\Rn\times[0,\rhoo]\to[0,2L]$ is a bounded function satisfying
\begin{description}
	\descitem{(VMO)}\label{bigvmo}
		$\lim_{\rho\searrow0}\mathbf{V}(\rho)=0,\qquad$where$\qquad \mathbf{V}(\rho):=\sup_{\xo\in\Omega,r\in(0,\rho]}\fint_{B_r(\xo)\cap\Omega}\mathbf{v}_{\xo}(x,r)\dx,$
\end{description}
where $\fint_S f(x) \dx = |S|^{-1}\int_{S}f(x)\dx$ is the average value of $f$ over the set $S$ with $|S|>0$.

Finally, we will assume the inhomogeneous term $b$ has \emph{controllable growth}, i.e.
\begin{description}
	\descitem{(I)}\label{vmocontrol}
		$b(x,\xi,z)\leq L(1 + |z|^2)^{p(x)-1}.$
\end{description}
\section{Statement of main result}

We the use method of $\mathcal{A}$-harmonic approximation to prove Theorem~\ref{vmoreg}. Developed in \cite{duzaar_steffen_2002} in the context of geometric measure theory, and \cite{duzaar_grotowski_2000} in a PDE setting, this method is a generalisation of de Giorgi's harmonic approximation lemma from minimal surface theory \cite{de_giorgi_1961}. This method has proved hugely successful in showing gradient continuity for general classes of problems in elliptic PDE \cite{beck_2007, beck_2011}, as well as in the variational setting \cite{beck_2011, duzaar_gastel_grotowski_2000, duzaar_grotowski_kronz_2004, duzaar_grotowski_kronz_2005, schmidt_2009}. Related generalisations have also been developed, including a recent Lipschitz-truncation technique that has allowed for a completely direct proof \cite{dldv}. We refer the reader to \cite{duzaar_mingione_2009} for review.

As stated in the previous section, we generalise the `$\varepsilon$-regularity' result from \cite{foss_mingione_2008} in multiple directions. In doing so, we obtain local analogue for systems which can have both $p(x)$ nonstandard growth, and VMO discontinuities. This improves and generalises the results found in \cite{habermann_2013} and \cite{bogelein_duzaar_habermann_scheven_2011}.

In order for a point to be singular - that is, to fail to be $\alpha$-H\"older continuous for some fixed $\alpha\in(0,1)$ - we require certain local energy functionals to exceed some energy quanta at all small scales. Although the regularity pertains to the solution $u$, the functionals considered have dependence on its gradient. In particular, the inhomogeneous nature of the higher integrability estimates in the variable exponent setting means that,  in contrast to \cite{foss_mingione_2008}, the constants obtained in the H\"older estimates are not uniform in the spatial variable. Instead, they depend in a critical way upon the gradient, in addition to the ellipticity and growth bounds of the system, as well as the chosen exponent function $p$, and of course the desired H\"older exponent $\alpha\in(0,1)$.

Since the quanta depend on the desired H\"older exponent, we give multiple characterisations of the singular sets. For fixed $\alpha\in(0,1)$, and $u\in\wpx$ define its $\alpha$-\emph{regular set} as
\begin{align*}
	\reg^\alpha_u(\Omega) &:= \left\{\, x\in\Omega : u \in C^{0,\alpha}(N(x),\Rn)\text{ for some open neighbourhood }N \text{ of } x\,\right\},
\end{align*}
and its regular set as
\begin{align*}
	\reg_u(\Omega) &:= \left\{\, x\in\Omega : u \in C^{0,\alpha}(N(x),\Rn)\text{ for every } \alpha\in(0,1)\,\right\}.
\end{align*}

Similarly, we consider the $\alpha$-\emph{singular set} and \emph{singular set} of $u$ as
\begin{align*}
	\sing^\alpha_u(\Omega) &:= \Omega\setminus\reg^\alpha_u(\Omega)\qquad\text{and}\qquad \sing_u(\Omega) := \Omega\setminus\reg_u(\Omega).
\end{align*}
We resolve $\sing^\alpha_u(\Omega)\subset \left(\Sigma^\kappa_{1,\Omega}\cup\Sigma^\sigma_{2,\Omega}\cup \Sigma_{3,\Omega}\right)$, where
\begin{align}\label{sigmaonevmo}
	\Sigma^\kappa_{1,\Omega} &:= \left\{\, \xo\in\Omega : \liminf_{\rho\downarrow0}\fint_{\brhoxo}|Du - \dnu|\dx\geq\kappa\right\},
\end{align}
and
\begin{align}\label{sigmatwovmo}
	\Sigma^\sigma_{2,\Omega} &:= \left\{\, \xo\in\Omega : \liminf_{\rho\downarrow0}\rho\fint_{\brhoxo}|Du|\dx\geq\sigma \,\right\},
\end{align}
and
\begin{align}\label{sigmathreevmo}
	\Sigma_{3,\Omega} &:= \left\{\, \xo\in\Omega : \limsup_{\rho\downarrow0}\big|\dnu\big|= \infty \,\right\},
\end{align}
for some $\kappa$ and $\sigma$, which we can in principle calculate explicitly, satisfying $\lim_{\alpha\to1}\kappa,\sigma = 0$. We similarly characterise $\sing_u(\Omega)\subset\left(\Sigma^0_{1,\Omega}\cup\Sigma^0_{2,\Omega}\right)$, where we the inequalities in \eqref{sigmaonevmo} and \eqref{sigmatwovmo} become strict.

Using this characterisation of the singular sets, we obtain the following theorem.

\begin{theorem}\label{vmoreg}
Let $u\in W^{1,p(\cdot)}\big(\Omega,\RN\big)$ be a weak solution to \eqref{vmoweaka} under assumptions \emph{\vmoellipticity--\vmo}, where the inhomogeneity $b$ satisfies \emph{\vmocontrol}. Then there exist $\kappa,\sigma>0$ such that the following hold:
\begin{enumerate}[(i)]
	\item $\reg_u(\Omega)$ and $\reg^\alpha_u(\Omega)$ are relatively open in $\Omega$, for each $\alpha\in(0,1),$
	\item $u\in C^{0,\alpha}\big(\reg_u(\Omega),\RN\big)$ for every $\alpha \in (0,1)$, and \\
		for every $\alpha \in (0,1)$ $u\in C^{0,\alpha}\big(\reg^\alpha_u(\Omega),\RN\big),$
	\item $\sing_u(\Omega)\subset \Big(\Sigma^0_{1,\Omega}\cup\Sigma^0_{2,\Omega}\cup \Sigma_{3,\Omega}\Big)$, and 
		$\sing^\alpha_u(\Omega)\subset \Big(\Sigma^\kappa_{1,\Omega}\cup\Sigma^\sigma_{2,\Omega}\cup \Sigma_{3,\Omega}\Big)$, with these sets defined in \eqref{sigmaonevmo}, \eqref{sigmatwovmo} and \eqref{sigmathreevmo}.
\end{enumerate}
In particular, we have $\mathcal{L}^n(\sing_u(\Omega)) = 0$.
\end{theorem}

\section{Preliminary Tools}
We now introduce some basic tools, most of which are standard in these types of arguments.

\subsection{ The function V}

We will frequently refer to the function $V \equiv V_{p}:\Rm\to\Rm$ satisfying $V(\xi) := \left(1 + |\xi|^2\right)^{\frac{p - 2}{4}}\xi$ for each $\xi\in\Rm$ and fixed $1<p<\infty$. This function is a standard tool when studying these types of problems, and has a number of favourable algebraic properties. Note that the constants appearing in the following lemma have continuous dependences on the exponents. Unless stated otherwise, we always take $p = p_2 = \sup_{B_{\rho}(\xo)}p(x)$.

\begin{lemma}\label{vproperties}
Let $1<p<\infty$ and $V \equiv V_{p}:\Rm\to\Rm$ be the function defined above. Then for any $\xi,\eta\in\Rm$ and $t>0$ there holds
\begin{align*}
\begin{array}{rll}
	(i)& |V(t\xi)| \leq \max\{t,t^{p}\}|V(\xi)|;&\\
	(ii)& |V(\xi + \eta)| \leq c(p)\left(|V(\xi)| + |V(\eta)|\right);&\\
	(iii)& c(p,m)|\xi - \eta|\leq\frac{|V(\xi) - V(\eta)|}{(1 + |\xi|^2 + |\eta|^2)^{\frac{p-2}{4}}}\leq c(m,p)|\xi - \eta|;&\\
	(iv)& 2^{\frac{2-p}{4}}\min\{|\xi|,|\xi|^{\frac{p}{2}}\} \leq |V(\xi)| \leq \min\{|\xi|,|\xi|^{\frac{p}{2}}\}&  \text{if }  \  1<p<2;\\
	\ &	  \max\{|\xi|,|\xi|^{\frac{p}{2}}\} \leq |V(\xi)| \leq 2^{\frac{2-p}{4}}\max\{|\xi|,|\xi|^{\frac{p}{2}}\}\qquad& \text{if } \ 2\leq p<\infty;\\
\end{array}
\end{align*}
\end{lemma}

The following improvement of Lemma~2.2 from \cite{acerbi_fusco} is taken from Lemma~2.3 in \cite{bogelein_duzaar_habermann_scheven_2012}.
\begin{lemma}\label{improv}
Let $1<\gamma_1\leq p\leq\gamma_2<\infty$, $\xi,\eta\in\Rm$, and $t>0$. Then for $c=c(\gamma_1,\gamma_2)$ there holds
\begin{align*}
	c^{-1}(1 + |\eta| + |\xi|)^{p-2}|\xi-\eta|^2\leq (1 + |\eta|)^{p_2-2}\left|V\left(\frac{\xi-\eta}{1 + |\eta|}\right)\right|^2\leq c(1 + |\eta| + |\xi|)^{p-2}|\xi-\eta|^2.
\end{align*}
\end{lemma}

We will make use of the following iteration lemma, a standard tool when considering decay estimates. This version is taken from Lemma~7.3 in the classical text \cite{giusti} via \cite{foss_mingione_2008}.
\begin{lemma}[Iteration Lemma]\label{iteration}
Let $f:[0,\rho]\to\R$ be a positive, nondecreasing function satisfying
\begin{align*}
	f(\theta^{k+1}\rho)\leq \theta^{\gamma}f(\theta^k\rho) + b\theta^k\rho
\end{align*}
for every $k\in\N$, where $\theta\in(0,1)$, and $\gamma\in(0,n)$, $b\geq0$. Then there exists a constant $c = c(\theta,\gamma,n)$ such that for every $r\in(0,\rho)$ there holds
\begin{align*}
	f(r)\leq c\left[\left(\frac{r}{\rho}\right)^\gamma f(\rho) + bt^\gamma\right].
\end{align*}
\end{lemma}

We estimate the Hausdorff dimension of the singular sets by use of the following result of Giusti \cite{giusti}.
\begin{lemma}\label{dimred}
Let $A$ be an open subset of $\Rn$ and $\lambda$ a non-negative and increasing finite set function, defined on the family of open subsets of $A$ that is countably superadditive. That is,
\begin{align*}
	\sum_{i\in\mathbb{N}}\lambda\left(\mathcal{O}_i\right)\leq \lambda\left(\bigcup_{i\in\mathbb{N}}\mathcal{O}_i\right)
\end{align*}
whenever $\{\mathcal{O}_i\}_{i\in\mathbb{N}}$ is a countable family of pairwise disjoint open subsets of $A$. Then for $0<\tau<n$ there holds
\begin{align*}
	\dimh(E^\tau)\leq\tau,
\quad\text{where}\quad
	E^\tau:=\bigg\{x\in A : \limsup_{\rho\downarrow0}\rho^{-\tau}\lambda\left(B_\rho(x)\right)>0\bigg\}.
\end{align*}
\end{lemma}

We will treat $V^2$ as if it were convex, since it can be estimated from above and below by the convex map $z\mapsto(1 + |z|)^{p-2}|z|^2$ (see Definition~6.1 in \cite{schmidt_2008}). For $x\in\R^k$ and some $c>1$ there holds
\begin{align}\label{almostconvex}
	c^{-1}|V(z)|^2\leq(1 + |z|)^{p-2}|z|^2\leq c|V(z)|^2.
\end{align}

When decomposing $V$, we use the following estimate, which can be retrieved from \cite{campanato_82a}. 

\begin{lemma}\label{camp82}
Given $\xi,\eta\in\R^k$ and $q>-1$, there exist constants $c_1(q),c_2(q)>0$ such that
\begin{align*}
	c_1(q)(1 + |\xi| + |\eta|)^q\leq\int_{0}^{1}(1 + |\xi + t\eta|)^q\dt \leq c_2(1 + |\xi| + |\eta|)^q.
\end{align*}
\end{lemma}

\subsection{Affine maps}

For $\xo\in\RN, \rho>0$, and $u\in L^2(\brhoxo,\RN)$, we have a unique affine function $\ell_{\xo,\rho}:\Rn\to\RN$ of the form
\begin{align*}
	\lrhoxo(\xo) + D\lrhoxo(x-\xo)
\qquad\mbox{that minimises}\qquad
\ell\mapsto\fint_{\brhoxo}|u-\ell|^2\dx.
\end{align*}
Here, $\lrhoxo(\xo)\in\RN$ and $D\lrhoxo\in\Hom(\Rn\otimes\RN)$. Indeed, by direct calculation there holds
\begin{align}\label{lform}
	\lrhoxo(\xo) = \urho
\qquad\mbox{and}\qquad
D\lrhoxo = \frac{n+2}{\rho^2}\fint_{\brhoxo}u\otimes(x-\xo)\dx.
\end{align}

To compare minimising affine functions on concentric balls, we have the following (see Lemma~2 in \cite{kronz_2002} for the case $p=2$, Lemma~2.2 from \cite{habermann_2013} for general $p$).
\begin{lemma}\label{affinegrads}
Fix $p\geq1$, $\theta\in(0,1)$, and $u\in L^p(\brhoxo,\RN)$. Denote by $\lrhoxo$ and $\lthetarhoxo$ the minimising affine functions on balls of radius $\rho$ and $\theta\rho$ of the form \eqref{lform}. Then we can estimate
\begin{align}\label{dlest1}
	|D\lrhoxo - D\lthetarhoxo|^p\leq \left(\frac{n+2}{\theta\rho}\right)^p\fint_{\bthetarhoxo}|u-\lrhoxo|^p\dx.
\end{align}
More generally, there holds for any affine $\Upsilon:\Rn\to\RN$
\begin{align}\label{dlest2}
	|D\lrhoxo - D\Upsilon|^p\leq \left(\frac{n+2}{\rho}\right)^p\fint_{\brhoxo}|u-\Upsilon|^p\dx.
\end{align}
\end{lemma}

The following quasi-minimisation property of $\lrhoxo$ appears as Corollary~2.4 in \cite{habermann_2013}.
\begin{lemma}\label{minil}
Fix $p\geq1$, $\lambda>0$, $u\in L^p(\brhoxo,\RN)$ and take $\lrhoxo$ defined in \eqref{lform}. Then for all affine $\Upsilon:\Rn\to\RN$ there holds
\begin{align*}
	\fint_{\brhoxo}|u-\lrhoxo|^p\dx\leq c\fint_{\brhoxo}|u-\Upsilon|^p\dx,
\end{align*}
and
\begin{align*}
	\fint_{\brhoxo}\left|V\left(\lambda\left(u-\lrhoxo\right)\right)\right|^2\dx\leq c\fint_{\brhoxo}\left|V\left(\lambda\left(u-\Upsilon\right)\right)\right|^2\dx.
\end{align*}
Here, the constant $c$ depends only on $n$ and $p$, and the dependence on $p$ is continuous.
\end{lemma}
A corollary of Lemma~\ref{affinegrads} is the following, adapted from \cite{habermann_2013} and proved in \cite{vdh_phd}.
\begin{corollary}\label{gradscales}
Fix $p\geq1$, $\theta\in(0,1)$ and $u\in L^p(\brhoxo,\RN)$. Denote by $\lrhoxo$ and $\lthetarhoxo$ the minimising affine functions on concentric balls of radius $\rho$ and $\theta\rho$ of the form \eqref{lform}. If the smallness condition
\begin{align*}
	\fint_{B_{\rho}(\xo)}\left|V\left(\frac{u - \lrhoxo}{\rho(1 + |D\lrhoxo|)}\right)\right|^2\dx\leq\left(\frac{1}{4}\frac{\theta^{n+1}}{n+2}\right)^2
\quad\text{holds, then}\quad
	\frac{1 + |D\lrhoxo|}{1 + |D\lthetarhoxo|} &\leq 2.
\end{align*}
\end{corollary}
\begin{proof}[Proof of Corollary~\ref{gradscales}:]
We can estimate via \eqref{dlest2} with $p=1$
\begin{align}
	1 + |D\lrhoxo| &\leq 1 + |D\lthetarhoxo| + |D\lrhoxo - D\lthetarhoxo|\nonumber\\
		&\leq 1 + |D\lthetarhoxo| + \left(\frac{n+2}{\theta\rho}\right)\fint_{\bthetarhoxo}|u-\lrhoxo|\dx\nonumber\\
		&\leq 1 + |D\lthetarhoxo| + \left(\frac{n+2}{\theta^{n+1}}\right)\fint_{\brhoxo}\left|\frac{u-\lrhoxo}{\rho(1 + |D\lrhoxo|)}\right|\dx(1 + |D\lrhoxo|).\label{gradscaleint}
\end{align}
Now, when $p\geq2$ we can calculate via H\"older's inequality, \vprop$(iii)$ and the smallness condition
\begin{align*}
	\fint_{\brhoxo}\left|\frac{u-\lrhoxo}{\rho(1 + |D\lrhoxo|)}\right|\dx &\leq \left(\fint_{\brhoxo}\left|\frac{u-\lrhoxo}{\rho(1 + |D\lrhoxo|)}\right|^2\dx\right)^{\frac{1}{2}}\\
		&\leq \left(\fint_{\brhoxo}\left|V\left(\frac{u-\lrhoxo}{\rho(1 + |D\lrhoxo|)}\right)\right|^2\dx\right)^{\frac{1}{2}}\\
		&\leq \frac{1}{2}\left(\frac{\theta^{n+1}}{n+2}\right).
\end{align*}
Writing $T_+ = \{x\in\brhoxo:u-\lrhoxo\geq\rho(1 + |D\lrhoxo|)\}$ and similarly define the set $T_- = \{x\in\brhoxo:u-\lrhoxo\leq\rho(1 + |D\lrhoxo|)\}$, for $1<p<2$ we can use the same reasoning to find
\begin{align*}
	\fint_{\brhoxo}&\left|\frac{u-\lrhoxo}{\rho(1 + |D\lrhoxo|)}\right|\dx \\
		&\leq \left(\fint_{\brhoxo}\left|\frac{u-\lrhoxo}{\rho(1 + |D\lrhoxo|)}\right|^2\chi_{_{T_-}}\dx\right)^{\frac{1}{2}}
			+\left(\fint_{\brhoxo}\left|\frac{u-\lrhoxo}{\rho(1 + |D\lrhoxo|)}\right|^q\chi_{_{T_+}}\dx\right)^{\frac{1}{q}}\\
		&\leq \left(\fint_{\brhoxo}\left|V\left(\frac{u-\lrhoxo}{\rho(1 + |D\lrhoxo|)}\right)\right|^2\chi_{_{T_-}}\dx\right)^{\frac{1}{2}}
			+\left(\fint_{\brhoxo}\left|V\left(\frac{u-\lrhoxo}{\rho(1 + |D\lrhoxo|)}\right)\right|^2\chi_{_{T_+}}\dx\right)^{\frac{1}{q}}\\
		&\leq \frac{1}{2}\left(\frac{\theta^{n+1}}{n+2}\right).
\end{align*}
Plugging these estimates into \eqref{gradscaleint} concludes the proof.
\end{proof}

\subsection*{Excess Functionals}
We make use of the following renormalised zeroth-order and first-order functionals. For any $\xo\in\Omega$ with $\brhoxo\subset\subset\Omega$, fixed $u\in u\in W^{1,p(\cdot)}(\Omega,\RN)$, and affine map $\ell:\Rn\to\RN$ with $D\ell\in\RnN$, we set
\begin{align*}
	\Psi(\xo,\ell,\rho)&:=\fint_{\brhoxo}\bigg|V\bigg(\frac{u - \ell}{\rho(1 + |D\ell|)}\bigg)\bigg|^2\dx,\qquad\text{and}\qquad
	\Phi(\xo,D\ell,\rho):=\fint_{\brhoxo}\bigg|V\bigg(\frac{Du - D\ell}{1 + |D\ell|}\bigg)\bigg|^2\dx.
\end{align*}

\subsection*{The $\mathcal{A}$-harmonic approximation lemma}
The workhorse of this theorem is the $\A$-harmonic approximation lemma. First appearing in the literature in \cite{duzaar_grotowski_2000} in the context of regularity theory for partial differential equations, this version is taken from \cite{beck_2011}.

Let $\A:\Hom(\Rn\otimes\RN)\times \Hom(\Rn\otimes\RN)\to\R$ be an elliptic bilinear form with constant coefficients. We say a function $h\in W^{1,1}(\brhoxo,\RN)$ is $\mathcal{A}$-harmonic if for all $\phi\in C^1_0(\brhoxo,\RN)$,
		\begin{align*}
			\fint_{\brhoxo}\A(Dh,D\phi)\, dx &= 0.
		\end{align*}

\begin{lemma}[$\A$-harmonic approximation]\label{aharmonic}
Fix $0<\nu\leq L <\infty$, $1<p_2<\infty$, and let $\A:\Hom(\Rn\otimes\RN)\times \Hom(\Rn\otimes\RN)\to\R$ be a bilinear form, which is elliptic in the sense of Legendre-Hadamard with ellipticity constant $\nu$ and upper-bound $L$. Given $\varepsilon>0$ there exists a $\delta_1 = \delta_1(n,N,p,L/\nu,\varepsilon)$ such that for all $\kappa\in(0,1]$ and every $w\in W^{1,p_2}(\brhoxo,\RN)$ satisfying
		\begin{align*}
			\fint_{\brhoxo}|V(Dw)|^2\, dx &\leq \kappa^2,\qquad\text{and}\qquad
			\left|\fint_{\brhoxo}\A(Dw,D\phi)\, dx\right|\leq \kappa\delta\sup_{\brhoxo}|D\phi|
		\end{align*}
for all $\phi\in C^1_0(\brhoxo,\RN)$, there exists an $\A$-harmonic function $h\in W^{1,p_2}(\brhoxo,\RN)$ satisfying
		\begin{align}\label{harmpriori}
			\sup_{\bhalfrhoxo}\left(|Dh| + \rho|D^2h|\right)\leq c_h\quad\mbox{and}\quad\fint_{\bhalfrhoxo}\left|V\left(\frac{w - \kappa h}{\rho}\right)\right|^2\, dx\leq\kappa^2\varepsilon.
		\end{align}
\end{lemma}
Here $V = V_{p_2}$, and the constant $c$ depends only on $n,N,p_2,$ and the ratio $\frac{L}{\nu}$.

\section{Higher Integrability}
Partial continuity results such as Theorem~\ref{vmoreg} require the higher integrability of the gradient of weak solutions to \eqref{vmoweaka}. The proof of these results requires the usual $\log$-H\"older condition and energy bounds \eqref{vmomodfinite}, as well as the following restrictions on the domain.

Taking $\xo\in\Omega$, fix $\delta>0$ (to be determined in a moment) and $\rho_0>0$ such that
\begin{align}\label{rhomega}
\begin{cases}\begin{array}{rll}
\omega_p(8n\rho_0)&\leq \sqrt{\frac{n + 1}{n}} - 1&\\
 0\leq\omega_p(\rho)\log\left(\frac{1}{\rho}\right)&\leq L	\quad &\mbox{for all } \rho\leq\rho_0\\
\omega_p(\rho)&\leq\frac{1}{4}\delta	\quad &\mbox{for all } \rho\leq\rho_0.
\end{array}\end{cases}
\end{align}

Now for each $\rho<\rhoo$, define
\begin{align*}
	p_1(\xo) := \inf_{\brhoxo} p(x),\qquad\text{and}\qquad p_2(\xo) := \sup_{\brhoxo} p(x),
\end{align*}
and set
\begin{align*}
	p_m(\xo) := \inf_{B_{\rhoo}(\xo)} p(x),\qquad\text{and}\qquad p_M(\xo) := \sup_{B_{\rhoo}(\xo)} p(x).
\end{align*}
To ease notation we will suppress the dependence on $\xo,\rho$ in these quantities and consider only a model case. Now \eqref{rhomega} lets us calculate
\begin{align*}
	 p_2-p_1\leq p_M-p_m\leq\omega_p(\rhoo)\leq\frac{1}{4}\delta.
\end{align*}
Hence
\begin{align}\label{guize}
	p_2(1 + \frac{1}{2}\delta)\leq p_M(1 + \frac{1}{2}\delta)\leq p_m(1 + \delta) \leq p_1(1 + \delta) \leq p(x)(1 + \delta),
\end{align}
for all $x\in B_{\rhoo}(\xo)$ and $0<\rho\leq\rhoo$. In view of \eqref{guize} we can suppress dependence on $p_2$, since we always consider $\rho<\rhoo$.

We are now in a position to state the higher integrability result for $Du$. First proved by Zhikov in \cite{zhikov_1997}, the version that we will use is adapted from Remark~3.1 in \cite{habermann_2013}.

\begin{lemma}\label{frozenexponent}
Let $u\in W^{1,1}(\Omega,\RN)$ with $Du\in L^{p(\cdot)}\big(\Omega,\Hom(\Rn\otimes\RN)\big)$ and $p:\Omega\to\R$ satisfying \eqref{vmomodfinite} be a weak solution to \eqref{vmoweaka}, where $a$ satisfies \emph{\vmoellipticity--\vmogrowth} and \emph{\vmocontinu}, $b$ satisfies \emph{\vmocontrol}. Then there exists a $\tilde\rho>0$ and a $\delta>0$ such that for any $\theta\in(0,\frac{1}{2})$, $p_0\in[p_1,p_2(1 + \frac{\delta}{2})]$ and $p\in[p_1,p_2]$, whenever $B_{2\rho}(\xo)\subset B_{\tilde\rho}(\xo)\subset\Omega$, there holds
\begin{align*}
	\left(\fint_{B_{\theta\rho}(\xo)}|Du|^{p_0}\dx\right)^{\frac{1}{p_0}}\leq c\left(\fint_{B_{\rho}(\xo)}1 + |Du|^{p(x)}\dx\right)^{\frac{1}{p}},
\end{align*}
where $c = c(n,N,\gamma_1,\gamma_2,L/\nu,E,\theta)$, $\tilde\rho = \tilde\rho(n,N,\gamma_1,\gamma_2,L/\nu,E,\omega_p)$ and $\delta = \delta(n,N,\gamma_1,\gamma_2,L/\nu,E)$. The dependence on $\theta$ is such that the constant blows up as $\theta \to 0$. Note that we could take $\theta\in(0,1)$ if we allow the constant to blow up as $\theta\to1$.
\end{lemma}

If necessary, we restrict $\rhoo\leq\tilde\rho$. The following corollary allows us to obtain gradient estimates for our solution in terms of affine functions $\ell$, under a certain smallness assumption. This version is immediate from Lemma~3.3 found in \cite{habermann_2013}, with obvious modifications.
\begin{corollary}\label{affinest}
Let $B_{2\rho}(\xo)\subset B_{\rhoo}(\xo)\subset\Omega$, $\ell:\Rn\to\RN$ be an arbitrary affine function, and let $u$ satisfy the conditions of Lemma~\ref{frozenexponent}. Then for any $\theta\in(0,2]$ and $p\leq p_2$ there holds 
\begin{align*}
(i)	\qquad\fint_{B_{\theta\rho}(\xo)}|Du|^{p}\dx\leq 2^{p_2+1}\left(1 + |D\ell|\right)^{p}\qquad\mbox{whenever}\qquad\Phi(\xo,D\ell,\theta\rho)\leq1,\qquad
\end{align*}
and
\begin{align*}
(ii) \qquad\qquad\	\frac{1}{2}\leq\frac{1 + |(Du)_{\theta\rho,\xo}|}{1 + |D\ell|}\leq 3\qquad\mbox{whenever}\qquad\Phi(\xo,D\ell,\theta\rho)\leq\frac{1}{36}.\qquad\qquad\
\end{align*}
\end{corollary}

We will also use the following interpolation estimate for $L^p$-functions, which allows us to use the log-convexity of $L^p$-norms to equivalently consider the Lebesgue points of our solution in different $L^p$-spaces. This version appears as $(7.9)$ in \cite{gt}.

\begin{lemma}\label{logconvex}
Let $u\in L^{p}\cap L^{q}(\Omega,\R^k)$ for $0<p< q\leq\infty$. Then $u\in L^{s}(\Omega,\R^k)$ for all $p\leq s\leq q$, and there holds
\begin{align*}
	\|u\|_{L^s}\leq\|u\|_{L^p}^{\theta}\|u\|_{L^q}^{1-\theta},
\end{align*}
where $s$ satisfies $\frac{1}{s} = \frac{\theta}{p} +\frac{1-\theta}{q}$.
\end{lemma}
An estimate similar to the following can be be retrieved from the proof of Lemma~3.4 in \cite{bogelein_duzaar_habermann_scheven_2012}. We adapt their proof to our choice of exponent.

\begin{lemma}\label{low}
	Let $M>0$ and assume that $u\in\wpx$ is a weak solution to \eqref{vmoweaka} under structure conditions \emph{\vmoellipticity--\vmo}, with the inhomogeneity satisfying \emph{\vmocontrol}, and $p:\Omega\to\R$ satisfying \eqref{vmomodfinite}. Fix $D\ell\in\RnN$ to satisfy $\Phi(\xo,D\ell,\theta\rho)\leq\frac{1}{36}$ and $|D\ell|<M$, and take some smooth $\eta\in C^{\infty}_c(\brhoxo)$. Then
\begin{align*}
	\fint_{\brhoxo}&\eta^{\hat{p}}\Big|\big(1 + |D\ell|^2 + |Du|^2\big)^{\frac{p(x)-2}{2}} - \big(1 + |D\ell|^2 + |Du|^2\big)^{\frac{p_2-2}{2}}\Big|\big|Du - D\ell\big|^2\dx\\ 
	&\leq \frac{1}{2}\int_{\brhoxo}\eta^{\hat{p}}\big(1 + |D\ell|^2 + |Du|^2\big)^{\frac{p_2-2}{2}}\big|Du-D\ell\big|^2\dx + c\big(1 + |D\ell|\big)^{p_2}\omega_p(\rho).
\end{align*}	
Here, the constant depends on $n,N,L/\nu,\gamma_1,\gamma_2,M,\omega_p,E$ and $\omega_p$.
\end{lemma}

\begin{proof}[Proof of Lemma~\ref{low}:]
Noting that $p(x)\leq p_2$ in the domain, with $p_2 - p(x)\leq\omega(\rho)$, we have for $y\geq0$ the pointwise estimate
	\begin{align*}
	\Big|\big(1 + y\big)^{\frac{p(x)-2}{2}} - \big(1 + y\big)^{\frac{p_2-2}{2}}\Big| &= \bigg|\frac{(p(x) - p_2)}{2}\int_{0}^{1}\big(1 + y\big)^{\frac{sp(x) + (1-s)p_2-2}{2}}\log(1 + y)\, ds\bigg|\\
		&\leq \frac{1}{2}\omega_p(\rho)(1 + y)^{\frac{p_2-2}{2}}\log(1 + y).
	\end{align*}
	Consequently, we have via Young's inequality with exponents $(2,2)$
	\begin{align*}
	\fint_{\brhoxo}&\eta^{\hat{p}}\Big|\big(1 + |D\ell|^2 + |Du|^2\big)^{\frac{p(x)-2}{2}} - \big(1 + |D\ell|^2 + |Du|^2\big)^{\frac{p_2-2}{2}}\Big|\big|Du - D\ell\big|^2\dx\\
		&\leq c\omega_p(\rho)
		  \fint_{\brhoxo}\eta^{\hat{p}}\big(1 + |D\ell|^2 + |Du|^2\big)^{\frac{p_2-2}{4}+\frac{p_2}{2}}\log\big(1 + |D\ell|^2 + |Du|^2\big)\big|Du - D\ell\big|\dx\\
		&\leq \frac{1}{2}\fint_{\brhoxo}\eta^{\hat{p}}\big(1 + |D\ell|^2 + |Du|^2\big)^{\frac{p_2-2}{2}}|Du - D\ell|^2\dx
		\nl + c\omega_p(\rho)\fint_{\brhoxo}\eta^{\hat{p}}(1 + |D\ell| + |Du|)^{p_2}\log\big(1 + |D\ell| + |Du|\big)\dx\\
	\end{align*}
In handling the second term, when $|Du|\leq|D\ell|$, we easily have via the inequality $\log(1+|z|)\leq C(\delta)|z|^\delta$
	\begin{align*}
	\omega_p(\rho)\fint_{\brhoxo}&\eta^{\hat{p}}(1 + |D\ell| + |Du|)^{p_2}\log\big(1 + |D\ell| + |Du|\big)\chi_{(|Du|\leq|D\ell|)}\dx\\
		&\leq c(\delta)\omega_p(\rho)(1 + |D\ell|)^{p_2(1 + \delta)}
		\leq c(\delta,M)\omega_p(\rho)(1 + |D\ell|)^{p_2}.
	\end{align*}
Here we write $\chi_{(S)}$ to denote the characteristic function of the set $S$. On the other hand, when $|D\ell|<|Du|$ we again apply the inequality $\log(1+|z|)\leq C(\delta)|z|^\delta$, then Corollary~\ref{frozenexponent} and Corollary~\ref{affinest}~$(i)$ to find
	\begin{align*}
	\omega_p(\rho)\fint_{\brhoxo}&\eta^{\hat{p}}(1 + |D\ell| + |Du|)^{p_2}\log\big(1 + |D\ell| + |Du|\big)\chi_{(|D\ell|<|Du|)}\dx\\
		&\leq c(\delta)\omega_p(\rho)\fint_{\brhoxo}(1 + |Du|)^{p_2(1 + \frac{\delta}{4})}\dx\\
		&\leq c(\delta)\omega_p(\rho)\bigg(\fint_{\btworhoxo}(1 + |Du|)^{p_2}\dx\bigg)^{1 + \delta}\\		
		&\leq c(\delta,M)\omega_p(\rho)(1 + |D\ell|)^{p_2}.
	\end{align*}
Noting that the constant has inherited all of the dependences from Corollary~\ref{frozenexponent}, we combine these cases to conclude the result.\end{proof}

\section{ A Caccioppoli inequality}
The first key step is to establish a suitable Caccioppoli inequality.

\begin{lemma}[Caccioppoli Inequality]\label{eqcacc}
	Let $u\in\wpx$ be a weak solution to \eqref{vmoweaka} with \eqref{vmomodfinite}, under structure conditions \emph{\vmoellipticity--\vmo}, with the inhomogeneity satisfying \emph{\vmocontrol}. Fix an affine function $\ell:\Rn\to\RN$ satisfying $\Phi(\xo,D\ell,\rho)\leq\frac{1}{36}$ and $|D\ell|<M$.
Then there exist constants $\rho_{_{0}} = \rho_{_{0}}(n,N,L/\nu,\gamma_1,\gamma_2,\omega_p) <<1$ and $c = c_c(n,N,L/\nu,\gamma_1,\gamma_2,E,\omega_p)$ such that for every $\rho<\rho_{_{0}}$ and any ball $\brhoxo\subset\subset \Omega$ with $p_2 = \sup_{\brhoxo}p(\cdot)$ and $V=V_{p_2}$, the following estimate holds:
\begin{align*}
	\Phi(\xo,D\ell,\rho/2) \leq c_c\Bigg(\Psi(\xo,\ell,\rho)+  \omega_{\xi}^2\bigg(\fint_{\brhoxo}|u - \lx|\dx\bigg) + \mathbf{V}(\rho) + \omega_p(\rho) + \rho^2\Bigg).
\end{align*}
\end{lemma}

\begin{proof}[Proof of Lemma~\ref{eqcacc}:]
Taking a standard cutoff function $\eta\in C^{\infty}_0(\brhoxo)$ that satisfies $0\leq\eta\leq1$, $\eta = 1$ on $\bhalfrhoxo$, $\eta = 0$ outside $B_{\frac{3\rho}{4}}(\xo)$ and $|D\eta|\leq \frac{C}{\rho}$. We write $\phi := \eta^{\hat{p}}w$ for $\hat{p} = \max\{2,p_2\}$ and $w:=u - \lx$. Then $\phi\in W^{1,p(\cdot)}_0(\brhoxo,\RN)$, with
\begin{align}\label{derivthing}
	D\phi = \hat{p}\eta^{\hat{p} - 1}\big(u  - \lx\big)\otimes D\eta + \eta^{\hat{p}}(Du - D\ell).
\end{align}
Since $u$ solves \eqref{vmoweaka} we have
\begin{align*}
	\fint_{\brhoxo}a(x,u,Du)\cdot D\phi\dx \ =\ \fint_{\brhoxo}b(x,u,Du)\cdot \phi\dx,
\end{align*}
and trivially $\fint_{\brhoxo}\big(a(\cdot,\lx,D\ell)\big)_{\rho,\xo}\cdot D\phi\dx \ =\ 0.$ We calculate
\begin{align*}
	\text{I}&=\fint_{\brhoxo}\eta^{\hat{p}}\big[a(x,u,Du) - a(x,u,D\ell)\big]\cdot (Du - D\ell)\dx\\
		&= \fint_{\brhoxo}\Big[\big(a(\cdot,\lx,D\ell)\big)_{\rho,\xo} - a(x,\lx,D\ell)\Big]\cdot D\phi\dx
		\nl + \fint_{\brhoxo}\big[a(x,\lx,D\ell) - a(x,u,D\ell)\big]\cdot D\phi\dx
		\nl + \hat{p}\fint_{\brhoxo}\eta^{\hat{p}-1}\big[a(x,u,Du) - a(x,u,D\ell)\big]\cdot w\otimes D\eta\dx
		\nl + \fint_{\brhoxo}b(x,u,Du)\cdot  \eta^{\hat{p}}w\\
	 &= \text{II} + \text{III} + \text{IV} + \text{V},
\end{align*}
with the obvious labelling. We now consider each term independently. We combine elementary integration with \vmoellipticity\ and Lemma~\ref{camp82}, then Lemma~\ref{low} with Lemma~\ref{improv} to estimate
\begin{align*}
	\text{I}	&= \fint_{\brhoxo}\eta^{\hat{p}}\big[a(x,u,Du) - a(x,u,D\ell)\big]\cdot (Du - D\ell)\dx\\
		&= \fint_{\brhoxo}\eta^{\hat{p}}\int_{0}^{1}\Big[D_z a\big(x,u,D\ell + t(Du - D\ell)\big)(Du - D\ell)\Big]\cdot (Du - D\ell)\dt\dx\\
		&\geq c\fint_{\brhoxo}\eta^{\hat{p}}\int_{0}^{1}\big(1 + |D\ell + t(Du - D\ell)|\big)^{p(x) - 2}|Du - D\ell|^2\dt\dx\\
		&\geq c\fint_{\brhoxo}\eta^{\hat{p}}\big(1 + |D\ell|^2 + |Du|^2\big)^{p(x) - 2}|Du - D\ell|^2\dx\\
		&\geq c(1+|D\ell|)^{p_2}\fint_{\brhoxo}\eta^{\hat{p}} \bigg|V\bigg(\frac{Du-D\ell}{1 + |D\ell|}\bigg)\bigg|\dx - (1+|D\ell|)^{p_2}\omega^2_p(\rho).
\end{align*}
Here the first constant already has the dependences of the constant from Lemma~\ref{low}, and will not gain more dependences. The superquadratic and subquadratic cases differ only in the value of the constants.

Considering the second term, we calculate via \eqref{derivthing} that
\begin{align*}
	\text{II}
		&= \fint_{\brhoxo}\hat{p}\eta^{\hat{p}-1} \Big[\big(a(\cdot,\lx,D\ell)\big)_{\xo,\rho} - a(x,\lx,D\ell)\Big]\cdot w\otimes D\eta\dx
		\nl+ \fint_{\brhoxo}\eta^{\hat{p}} \Big[\big(a(\cdot,\lx,D\ell)\big)_{\xo,\rho} - a(x,\lx,D\ell)\Big]\cdot(Du - D\ell)\dx\\
		&= \text{II}_a + \text{II}_b,
\end{align*}
again with the obvious labelling.

Define the set $S_-:=\big\{x\in\brhoxo : |Du-D\ell|<1 + |D\ell|\big\}$ and $S_+:=\brhoxo\setminus S_-$. Now we can use \vmo\ along with Young's inequality (with exponent pairs $(2,2)$ and $(p_2,\frac{p_2}{p_2-1})$), keeping in mind $0\leq\mathbf{v}_{\xo}\leq 2L$, 
\vprop$(iv)$ and the bound $|D\ell|<M$ to compute
\begin{align*}
	\text{II}_b	&\leq \fint_{\brhoxo}\eta^{\hat{p}} \big|\big(a(\cdot,\lx,D\ell)\big)_{\xo,\rho} - a(x,\lx,D\ell)\big| |Du - D\ell|\dx\\
		&\leq c(1+|D\ell|)^{p_2}\big[1 + \log(1 + |D\ell|)\big]\fint_{\brhoxo}\eta^{\hat{p}}\mathbf{v}_{\xo}(x,\rho)\bigg|\frac{Du - D\ell}{(1 + |D\ell|)}\bigg|\dx\\
		&\leq c(\varepsilon)(1+|D\ell|)^{p_2}\fint_{\brhoxo}\mathbf{v}^{\frac{p_2}{p_2-1}}_{\xo}(x,\rho)\chi_{(S_+)} + \mathbf{v}^{2}_{\xo}(x,\rho)\chi_{(S_-)}\dx
		\nl + \varepsilon(1+|D\ell|)^{p_2}\fint_{\brhoxo}\eta^{\hat{p}}\bigg|\frac{Du - D\ell}{(1 + |D\ell|)}\bigg|^2\chi_{(S_-)} + \eta^{\hat{p}}\bigg|\frac{Du - D\ell}{(1 + |D\ell|)}\bigg|^{p_2}\chi_{(S_+)}\dx\\
		&\leq c(\varepsilon)(1+|D\ell|)^{p_2}\mathbf{V}(\rho) + \varepsilon(1+|D\ell|)^{p_2}\fint_{\brhoxo}\eta^{\hat{p}}\bigg|V\bigg(\frac{Du-D\ell}{1 + |D\ell|}\bigg)\bigg|\dx.
\end{align*}
Here we have left $\varepsilon$ to be chosen later, and have used that when $q>1$ and $\mathbf{v}\leq 2L$, we have
	$\mathbf{v}^{q}_{\xo}(x,\rho)\leq (2L)^{q-1}\mathbf{v}_{\xo}(x,\rho).$
We have also used from the proof of Corollary~\ref{low}, that
	$(1+|D\ell|)^{p_2}[1 + \log(1 + |D\ell|)] \leq c(1+|D\ell|)^{p_2}$.

We can estimate II$_a$ in the same way, considering
$T_-:=\{x\in\brhoxo : |w|<(1 + |D\ell|)\rho\}$ and $T_+:=\brhoxo\setminus T_-,$
and calculate via the same process
\begin{align*}
	\text{II}_a	&\leq \fint_{\brhoxo}\eta^{\hat{p}} \big|\big(a(\cdot,\lx,D\ell)\big)_{\xo,\rho} - a(x,\lx,D\ell)\big| \bigg|\frac{w}{\rho}\bigg|\dx\\
		&\leq c(1+|D\ell|)^{p_2}\big[1 + \log(1 + |D\ell|)\big]\fint_{\brhoxo}\eta^{\hat{p}}\mathbf{v}_{\xo}(x,\rho)\bigg|\frac{w}{(1 + |D\ell|)\rho}\bigg|\dx\\
		&\leq c(1+|D\ell|)^{p_2}\Bigg(\mathbf{V}(\rho) + \fint_{\brhoxo}\bigg|V\bigg(\frac{w}{(1 + |D\ell|)\rho}\bigg)\bigg|\dx\Bigg).
\end{align*}

To estimate III we notice that owing to \vmocontinu\ and \eqref{derivthing} there holds
\begin{align*}
	\text{III}
		&\leq c(1 + |D\ell|)^{p_2}\fint_{\brhoxo}\omega_{\xi}\big(|u - \lx|\big)\bigg[\bigg|\frac{w}{(1 + |D\ell|)\rho}\bigg| + \eta^{\hat{p}}\bigg|\frac{Du - D\ell}{1 + |D\ell|}\bigg|\bigg]\dx
		= \text{III}_a + \text{III}_b.
\end{align*}
with the obvious notation. Now when $2\leq p_2<\infty$ we can use Young's inequality with exponents $(2,2)$ and \vprop$(iv)$ to deduce
\begin{align*}
	\text{III}_a + \text{III}_b
		&\leq (1 + |D\ell|)^{p_2}\fint_{\brhoxo} c(\varepsilon)\omega_{\xi}^2\big(|u - \lx|\big) + c(\varepsilon)\bigg|V\bigg(\frac{w}{(1 + |D\ell|)\rho}\bigg)\bigg|^2.
\end{align*}
To treat the subquadratic case, we recall the definition of $S_-$ and $S_+$ from estimate II. We will consider III$_b$, with the calculations for III$_a$ being completely analogous, replacing $S_-$ with $T_-$, and $S_+$ with $T_+$ as we did in II. For $\big|\frac{Du - D\ell}{(1 + |D\ell|)}\big|<1$ we find using Young's inequality (with both exponent pairs $(2,2)$ and $(p_2,\frac{p_2}{p_2-1})$), and \vprop$(iv)$, keeping in mind that $\omega_{\xi}\leq 1$ and  $\frac{p_2}{p_2 - 1}>2$:
\begin{align*}
	\text{III}_b	
		&\leq c(\varepsilon)(1 + |D\ell|)^{p_2}\fint_{\brhoxo} \omega_{\xi}^2\big(|u - \lx|\big)\chi_{(S_-)}
		+ \omega_{\xi}^{\frac{p_2}{p_2 - 1}}\big(|u - \lx|\big)\chi_{(S_+)}\dx
		\nl+ \varepsilon(1 + |D\ell|)^{p_2}\fint_{\brhoxo} \eta^{4}\bigg|\frac{Du - D\ell}{1 + |D\ell|}\bigg|^2\chi_{(S_-)}
		 + \eta^{2p_2}\bigg|\frac{Du - D\ell}{1 + |D\ell|}\bigg|^{p_2}\chi_{(S_+)}\dx\\
		&\leq (1 + |D\ell|)^{p_2}\fint_{\brhoxo}c(\varepsilon)\omega_{\xi}\big(|u - \lx|\big) + \varepsilon\eta^{\hat{p}}\bigg|V\bigg(\frac{Du - D\ell}{1 + |D\ell|}\bigg)\bigg|^2\dx.
\end{align*}
Analogously, for III$_a$ we obtain
\begin{align*}
	\text{III}_a	
		&\leq c(1 + |D\ell|)^{p_2}\fint_{\brhoxo}\omega_{\xi}\big(|u - \lx|\big) + \bigg|V\bigg(\frac{w}{1 + |D\ell|}\bigg)\bigg|^2\dx.
\end{align*}
Combining these estimates and applying Jensen's inequality yields
\begin{align*}
	\text{III}	
		&\leq c(1 + |D\ell|)^{p_2}\bigg[\omega_{\xi}\bigg(\fint_{\brhoxo} |u - \lx|\dx\bigg) + \fint_{\brhoxo}\bigg|V\bigg(\frac{w}{(1 + |D\ell|)\rho}\bigg)\bigg|^2\dx\bigg]
		\nl+ \varepsilon(1 + |D\ell|)^{p_2}\fint_{\brhoxo}\eta^{\hat{p}}\bigg|V\bigg(\frac{Du - D\ell}{1 + |D\ell|}\bigg)\bigg|^2\dx.
\end{align*}

Estimating the next term in the case where $p_2\geq2$, we use \vmodifferentiability\ to compute
\begin{align*}
	\text{IV}	
		&\leq c\fint_{\brhoxo}\eta^{\hat{p}-1}\int_{0}^{1}\Big|D_z a\big({x},u,D\ell + t(Du - D\ell)\big)\Big||Du - D\ell||w||D\eta|\dt\dx\nonumber\\
		&\leq c\fint_{\brhoxo}\eta^{\hat{p}-1}\left(1 + |D\ell| + |Du - D\ell|\right)^{p_2-2}|Du - D\ell|\bigg|\frac{w}{\rho}\bigg|\dx\nonumber\\
\end{align*}
Now we apply Young's inequality (with exponent pairs $(2,2)$ and $(p_2,\frac{p_2}{p_2-1})$), and finally \vprop$(iv)$ to calculate
\begin{align}\label{d1}
	\text{IV}
		&\leq c\fint_{\brhoxo}\eta^{\hat{p}-1}(1 + |D\ell|)^{p_2-2}|Du - D\ell|\bigg|\frac{w}{\rho}\bigg| + \eta^{\hat{p}-1}|Du - D\ell|^{p_2-1}\bigg|\frac{w}{\rho}\bigg|\dx\nonumber\\
		&\leq c(1 + |D\ell|)^{p_2}\fint_{\brhoxo}\eta^{\hat{p}-1}\bigg[\bigg|\frac{Du - D\ell}{1 + |D\ell|}\bigg| + \bigg|\frac{Du - D\ell}{1 + |D\ell|}\bigg|^{p_2-1}\bigg]\bigg|\frac{w}{\rho(1 + |D\ell|)}\bigg|\dx\nonumber\\
		&\leq (1 + |D\ell|)^{p_2}\fint_{\brhoxo}\varepsilon\eta^{\hat{p}}\bigg|V\bigg(\frac{Du - D\ell}{1 + |D\ell|}\bigg)\bigg|^{2} + c(\varepsilon)\bigg|V\bigg(\frac{w}{(1 + |D\ell|)\rho}\bigg)\bigg|^2\dx.
\end{align}

The setting with $1<p_2<2$ is more delicate. We begin by applying \vmodifferentiability\ and Lemma~\ref{camp82} to compute
\begin{align*}
	\text{IV}	
		&\leq c\fint_{\brhoxo}\eta^{\hat{p}-1}\int_{0}^{1}\big|D_z a({x},u,D\ell + t(Du - D\ell)\big||Du - D\ell||w||D\phi|\dt\dx\\
		&\leq c\fint_{\brhoxo}\eta^{\hat{p}-1}(1 + |D\ell + t(Du - D\ell)|)^{p_2-2}|Du - D\ell|\bigg|\frac{w}{\rho}\bigg|\dx\\
		&\leq c\fint_{\brhoxo}\eta^{\hat{p}-1}(1 + |D\ell| + |Du - D\ell|)^{p_2-2}|Du - D\ell|\bigg|\frac{w}{\rho}\bigg|\dx.
\end{align*}
Recalling the sets defined as $T_{-} = \{x\in\brhoxo:|w|<(1+|D\ell|)\rho\}$ and $S_{-} = \{x\in\brhoxo:|Du-D\ell|<(1+|D\ell|)\}$, with $T_+ = \brhoxo\setminus T_-$ and $S_+ = \brhoxo\setminus S_-$, we now decompose the domain of integration into four parts
\begin{align}\label{d2}
	\text{IV}	&\leq c(1 + |D\ell|)^{p_2}\fint_{\brhoxo}\eta^{\hat{p}-1}\bigg|\frac{Du - D\ell}{1 + |D\ell|}\bigg|^{p_2-1}\bigg|\frac{w}{\rho(1 + |D\ell|)}\bigg|\chi_{(T_+\cap S_+)}\dx\nonumber
		\nl + c(1 + |D\ell|)^{p_2}\fint_{\brhoxo}\eta^{\hat{p}-1}\bigg|\frac{Du - D\ell}{1 + |D\ell|}\bigg|\bigg|\frac{w}{\rho(1 + |D\ell|)}\bigg|\chi_{(T_-\cap S_-)}\dx\nonumber
		\nl + c(1 + |D\ell|)^{p_2}\fint_{\brhoxo}\eta^{\hat{p}-1}\bigg|\frac{w}{\rho(1 + |D\ell|)}\bigg|^{p_2}\chi_{(T_+\cap S_-)}\dx\nonumber
		\nl + c(1 + |D\ell|)^{p_2}\fint_{\brhoxo}\eta^{\hat{p}-1}\bigg|\frac{Du - D\ell}{1 + |D\ell|}\bigg|^{p_2-1}\bigg|\frac{w}{\rho(1 + |D\ell|)}\bigg|\chi_{(T_-\cap S_+)}\dx\nonumber\\
		&= \text{IV}_a + \text{IV}_b + \text{IV}_c + \text{IV}_d
\end{align}
We first use Young's inequality and \vprop$(iv)$ to show
\begin{align*}
	\text{IV}_a	
		&\leq (1 + |D\ell|)^{p_2}\fint_{\brhoxo}\bigg[\varepsilon\eta^{\frac{p_2}{p_2-1}}\bigg|\frac{Du - D\ell}{1 + |D\ell|}\bigg|^{p_2} + c(\varepsilon)\bigg|\frac{w}{\rho(1 + |D\ell|)}\bigg|^{p_2}\bigg]\chi_{(T_+\cap S_+)}\dx\nonumber\\
		&\leq (1 + |D\ell|)^{p_2}\fint_{\brhoxo}\bigg[\varepsilon\eta^{\hat{p}}\bigg|V\bigg(\frac{Du - D\ell}{1 + |D\ell|}\bigg)\bigg|^{2} + c(\varepsilon)\bigg|V\bigg(\frac{w}{(1 + |D\ell|)\rho}\bigg)\bigg|^2\bigg] \chi_{(T_+\cap S_+)}\dx,
\end{align*}
and similarly
\begin{align*}
	\text{IV}_b	
		&\leq (1 + |D\ell|)^{p_2}\fint_{\brhoxo}\bigg[\varepsilon\eta^{\hat{p}}\bigg|\frac{Du - D\ell}{1 + |D\ell|}\bigg|^2 + c(\varepsilon)\bigg|\frac{w}{\rho(1 + |D\ell|)}\bigg|^2\bigg]\chi_{(T_-\cap S_-)}\dx\nonumber\\
		&\leq (1 + |D\ell|)^{p_2}\fint_{\brhoxo}\bigg[\varepsilon\eta^{\hat{p}}\bigg|V\bigg(\frac{Du - D\ell}{1 + |D\ell|}\bigg)\bigg|^{2} + c(\varepsilon)\bigg|V\bigg(\frac{w}{(1 + |D\ell|)\rho}\bigg)\bigg|^2\bigg]\chi_{(T_-\cap S_-)}\dx.
\end{align*}
For the third term we only need \vprop$(iv)$ to compute
\begin{align*}
	\text{IV}_c
		&\leq c(1 + |D\ell|)^{p_2}\fint_{\brhoxo}\bigg|V\bigg(\frac{w}{(1 + |D\ell|)\rho}\bigg)\bigg|^2\chi_{(T_+\cap S_-)} \dx.
\end{align*}
Finally we use the fact that $p<2$ implies $2(p-1)<p$ and equivalently $\frac{p}{p-1}>2$, together with Young's inequality and \vprop$(iv)$ to find
\begin{align*}
	\text{IV}_d
		&\leq  (1 + |D\ell|)^{p_2}\fint_{\brhoxo}\bigg[\varepsilon\eta^{\hat{p}}\bigg|\frac{Du - D\ell}{1 + |D\ell|}\bigg|^{2(p_2-1)} + c(\varepsilon)\bigg|\frac{w}{\rho(1 + |D\ell|)}\bigg|^2\bigg]\chi_{(T_-\cap S_+)}\dx\nonumber\\
		&\leq \varepsilon(1 + |D\ell|)^{p_2}\fint_{\brhoxo}\bigg[\varepsilon\eta^{\hat{p}}\bigg|V\bigg(\frac{Du - D\ell}{1 + |D\ell|}\bigg)\bigg|^{2} + c(\varepsilon)\bigg|V\bigg(\frac{w}{(1 + |D\ell|)\rho}\bigg)\bigg|^2\bigg]\chi_{(T_-\cap S_+)}\dx.
\end{align*}
Compiling $\text{IV}_a, \text{IV}_b, \text{IV}_c, \text{IV}_d$, we have
\begin{align}\label{d2}
	\text{IV}
			&\leq \varepsilon(1 + |D\ell|)^{p_2}\fint_{\brhoxo}\eta^{\hat{p}}\bigg|V\bigg(\frac{Du - D\ell}{1 + |D\ell|}\bigg)\bigg|^{2}\dx
			 + c(\varepsilon)(1 + |D\ell|)^{p_2}\fint_{\brhoxo}\bigg|V\bigg(\frac{w}{(1 + |D\ell|)\rho}\bigg)\bigg|^2 \dx.
\end{align}
Combining \eqref{d1} and \eqref{d2} gives that for all values of $p_2$
\begin{align*}
	\text{IV}	&\leq \varepsilon(1 + |D\ell|)^{p_2}\fint_{\brhoxo}\eta^{\hat{p}}\bigg|V\bigg(\frac{Du - D\ell}{1 + |D\ell|}\bigg)\bigg|^{2}\dx
		 + c(\varepsilon)(1 + |D\ell|)^{p_2}\fint_{\brhoxo}\bigg|V\bigg(\frac{w}{(1 + |D\ell|)\rho}\bigg)\bigg|^2 \dx.
\end{align*}

In the superquadratic case, we estimate term V using \vmocontrol, Young's inequality with exponent pair $(p_2,\frac{p_2}{p_2-1})$, Corollary~\ref{affinest}~$(i)$, and  \vprop$(iv)$
\begin{align*}
	\text{V}
		&\leq L\fint_{\brhoxo}\rho(1 + |Du|)^{p(x)-1}\eta^{\hat{p}}\bigg|\frac{w}{\rho}\bigg| \dx\\
		&\leq c(1 + |D\ell|)^{p_2}\fint_{\brhoxo}\rho + \bigg|\frac{w}{(1 + |D\ell|)\rho}\bigg|^{p_2} \dx\\
		&\leq c(1 + |D\ell|)^{p_2}\Bigg(\rho + \fint_{\brhoxo}\bigg|V\bigg(\frac{w}{\rho(1 + |D\ell|)}\bigg)\bigg|^2 \dx\Bigg).
\end{align*}
When $p$ is subquadratic, we consider two distinct cases. On the set $T_+$ we find the calculations are identical to the superquadratic case. On $T_-$, we use Corollary~\ref{affinest}~$(i)$ to find
\begin{align*}
	\text{V}	&\leq c\fint_{\brhoxo}\rho(1 + |Du|)^{p(x)-1}\bigg|\frac{w}{\rho}\bigg|\chi_{(T_-)} \dx\\
		&\leq c(1 + |D\ell|)\fint_{\brhoxo}\rho(1 + |Du|)^{p_2-1}\dx\\
		&\leq c(1 + |D\ell|)^{p_2}\rho,
\end{align*}
concluding the estimate.

Collecting our terms and choosing $\varepsilon$ small enough to be absorbed on the left, we normalise by $\left(1 + |D\ell|\right)^{p_2}$ to obtain
\begin{align*}
	\fint_{\bhalfrhoxo}\bigg|V\bigg(\frac{Du - D\ell}{(1 + |D\ell|)}\bigg)\bigg|^{2}\dx
		& \leq c\Bigg(\fint_{\brhoxo}\bigg|V\bigg(\frac{w}{(1 + |D\ell|)\rho}\bigg)\bigg|^2 \dx
		\nl
		 + \omega_{\xi}^{2}\bigg(\fint_{\brhoxo}|u - \lx|\dx\bigg)
		  + \mathbf{V}(\rho) + \rho^2 \Bigg),
\end{align*}
as required.\end{proof}

\begin{remark}\label{caccoro}
Note if we replace $\bhalfrhoxo$ with $\bthetarhoxo$ for some $\theta\in(0,1)$, by different choice of cutoff function we can obtain a similar estimate, with our constant now gaining dependence on $\theta$ and blowing up as $\theta\to1$ or $\theta\to0$.
\end{remark}
\section{ $\mathcal{A}$-harmonic approximation}
The second step in the proof is to show that the solution to our PDE lies close to a solutions of a family of related linear PDE.

\begin{lemma}[Approximate $\mathcal{A}$-harmonicity]\label{aharmy}
	Fix $M>0$ and assume that $u$ is a weak solution to \eqref{vmoweaka} with \eqref{vmomodfinite} under structure conditions \emph{\vmoellipticity--\vmo}\ with the inhomogeneity satisfying \emph{\vmocontrol}. Then there exists a constant $C = C(M, n,N,L/\nu,\gamma_1,\gamma_2,E,\omega_p)$ and a radius $\rho_0<<1$ such that whenever $\rho<\rho_0$ and $\Phi(\xo,D\ell,\rho)\leq\frac{1}{36}$ for some affine map $\ell:\Rn\to\RN$ satisfying $|D\ell|<M$, there holds
\begin{align*}
	\bigg|\fint_{\brhoxo}&\big(D_{z}a(\cdot,\lx,D\ell)\big)_{\xo,\rho}(Du-D\ell)\cdot D\varphi\dx\bigg|\\
		&\leq c_{_{1}}(1 + |D\ell|)^{p_2-1}\Bigg(\mu\Big(\sqrt{\Phi(\xo,D\ell,\rho)}\Big)\sqrt{\Phi(\xo,D\ell,\rho)} + \mathbf{V}(\rho) + \mathbf{V}^\frac{1}{2}(\rho) + \Phi(\xo,D\ell,\rho)
		\nl \qquad\qquad + \omega_{\xi}\bigg(\fint_{\brhoxo}|u - \lx|\dx\bigg) + \rho\Bigg)\|D\varphi\|_{C(\brhoxo,\RN)},
\end{align*}
for all $\varphi\in C^\infty_0(\brhoxo,\RN)$.
\end{lemma}

\begin{proof}[Proof of Lemma~\ref{aharmy}:]
Taking some $\varphi\in C^1_0(\brhoxo,\RN)$ with $\|D\varphi\|_{L^\infty(\brhoxo,\RN)}= 1$, we set $v = u - \ell$ and begin by noting
\begin{align*}
	\fint_{\brhoxo}&\big(D_{z}a(\cdot,\lx,D\ell)\big)_{\xo,\rho}Dv\cdot D\varphi\dx\\
		&= \fint_{\brhoxo}\int_{0}^{1}\Big[\big(D_{z}a(\cdot,\lx,D\ell)\big)_{\xo,\rho} - \big(D_{z}a(\cdot,\lx,D\ell + tDv)\big)_{\xo,\rho}\Big]Dv\cdot D\varphi\dt\dx
		\nl + \fint_{\brhoxo}\Big[\big(a(\cdot,\lx,Du)\big)_{\xo,\rho} - a(x,\lx,Du)\Big]\cdot D\varphi\dx
		\nl + \fint_{\brhoxo}\big[a(x,\lx,Du) - a(x,u,Du)\big]\cdot D\varphi\dx
		\nl + \fint_{\brhoxo}b(x,u,Du)\cdot\varphi\dx\\
		&= \text{I} + \text{II} + \text{III} + \text{IV},
\end{align*}
with the obvious labelling.

To estimate I we use the differentiability condition \vmoctsdifferentiability, which differs in the super and subquadratic cases. For $2\leq p_2$ we calculate pointwise via \vmoctsdifferentiability\, Lemma~\ref{camp82}, \vprop$(iv)$ on the sets $S_+$ and $S_-$, and the bound $\mu\leq1$
\begin{align*}
	\bigg|\int_{0}^{1}&\Big[\big(D_{z}a(\cdot,\lx,D\ell)\big)_{\xo,\rho} - \big(D_{z}a(\cdot,\lx,D\ell + tDv(x))\big)_{\xo,\rho}\Big]\dt Dv\cdot D\varphi\bigg|\\
		&\leq \int_{0}^{1}\bigg|\fint_{\brhoxo}D_{z}a(\cdot,\lx,D\ell) - D_{z}a(\cdot,\lx,D\ell + tDv(x))\dy\bigg|\dt|Dv||D\varphi|\\
		&\leq L\int_{0}^{1}\fint_{\brhoxo}\bigg|\mu\bigg(\frac{|Du(x)-D\ell|}{1 + |D\ell|}\bigg)\big(1 + |D\ell| + |D\ell + tDv(x)|\big)^{p_2-2}\bigg|\dy\dt|Dv|\\
		&\leq c\big(1 +|D\ell|\big)^{p_2-1}\mu\bigg(\frac{|Du-D\ell|}{1 + |D\ell|}\bigg)\bigg[\bigg|\frac{Du-D\ell}{1 + |D\ell|}\bigg| + \bigg|\frac{Du-D\ell}{1 + |D\ell|}\bigg|^{p_2-1}\bigg]\\
		&\leq c(1 + |D\ell|)^{p_2-1}\bigg[\mu\bigg(\frac{|Du(x)-D\ell|}{1 + |D\ell|}\bigg)\bigg|V\bigg(\frac{Du-D\ell}{1 + |D\ell|}\bigg)\bigg| + \bigg|V\bigg(\frac{Du-D\ell}{1 + |D\ell|}\bigg)\bigg|^2\bigg].
\end{align*}

On the other hand, when $1<p_2<2$ we find via \vmoctsdifferentiability\, Lemma~\ref{camp82} and the fact $\mu\leq1$
\begin{align*}
	\bigg|\int_{0}^{1}&\Big[\big(D_{z}a(\cdot,\lx,D\ell)\big)_{\xo,\rho} - \big(D_{z}a(\cdot,\lx,D\ell + tDv(x))\big)_{\xo,\rho}\Big]\dt Dv\cdot D\varphi\bigg|\\
		&\leq L\int_{0}^{1}\fint_{\brhoxo}\bigg|\mu\bigg(\frac{|Du(x)-D\ell|}{1 + |D\ell|}\bigg)\bigg[\frac{1 + |D\ell| + |D\ell + tDv(x)|}{(1 + |D\ell|)(1 + |D\ell + tDv(x)|)}\bigg]^{2-p_2}\bigg|\dy\dt|Dv|\\
		&\leq c\int_{0}^{1}\bigg|\mu\bigg(\frac{|Du-D\ell|}{1 + |D\ell|}\bigg)\bigg[\frac{1 + |D\ell| + |Dv|}{(1 + |D\ell + tDv|)}\bigg]^{2-p_2}\bigg|\dt|Dv|\\
		&\leq c(1 + |D\ell|)^{p_2-1}\mu\bigg(\frac{|Du-D\ell|}{1 + |D\ell|}\bigg)\bigg|\frac{Du - D\ell}{1 + |D\ell|}\bigg|\\
		&\leq c(1 + |D\ell|)^{p_2-1}\bigg[\mu\bigg(\frac{|Du-D\ell|}{1 + |D\ell|}\bigg)\bigg|V\bigg(\frac{Du-D\ell}{1 + |D\ell|}\bigg)\bigg|\chi_{S_-} + \bigg|V\bigg(\frac{Du-D\ell}{1 + |D\ell|}\bigg)\bigg|^2\chi_{S_+}\bigg]\\
\end{align*}
In either case, after integrating and applying H\"older's inequality, the right hand side is just
\begin{align*}
	c(1 + |D&\ell|)^{p_2-1}\fint_{\brhoxo}\bigg|V\bigg(\frac{Du-D\ell}{1 + |D\ell|}\bigg)\bigg|^2 + \mu\bigg(\bigg|V\bigg(\frac{Du-D\ell}{1 + |D\ell|}\bigg)\bigg|\bigg)\bigg|V\bigg(\frac{Du-D\ell}{1 + |D\ell|}\bigg)\bigg|\dx\\
	&\leq c(1 + |D\ell|)^{p_2-1}\Bigg[\Phi(\xo,D\ell,\rho)
		 + \Bigg(\fint_{\brhoxo}\mu^2\bigg(\bigg|V\bigg(\frac{Du-D\ell}{1 + |D\ell|}\bigg)\bigg|\bigg)\dx\Phi(\xo,D\ell,\rho)\Bigg)^{\frac{1}{2}}\Bigg].
\end{align*}
Keeping in mind the concavity of $\mu^2$, we apply Jensen's and H\"older's inequalities to conclude
\begin{align*}
	\text{I}&\leq c(1 + |D\ell|)^{p_2-1}\Big(\Phi(\xo,D\ell,\rho) + \mu\Big(\sqrt{\Phi(\xo,D\ell,\rho)}\Big)\sqrt{\Phi(\xo,D\ell,\rho)}\Big).
\end{align*}
We briefly note that when $|Du|\leq|D\ell|$, since $\log(1+|z|)\leq C(\delta)|z|^\delta$ we easily have
\begin{align*}
	(1 + |D\ell|)^{p_2-1}\mathbf{V}^\frac{1}{2}(\rho)\left(\fint_{\brhoxo}\log^\gamma(1 + |Du|)\dx\right)^\frac{1}{2}\nonumber
		&\leq c(1 + |D\ell|)^{p_2-1}\mathbf{V}^\frac{1}{2}(\rho).
\end{align*}
On the other hand, when $|D\ell|<|Du|$ the same estimate with Lemma~\ref{frozenexponent} and Corollary~\ref{affinest}~$(i)$ imply
\begin{align*}
	(1 + |D\ell|)^{p_2-1}\mathbf{V}^\frac{1}{2}(\rho)&\left(\fint_{\brhoxo}\log^\gamma(1 + |Du|)\dx\right)^\frac{1}{2}\nonumber\\
		&\leq (1 + |D\ell|)^{\frac{p_2-2}{2}}\mathbf{V}^\frac{1}{2}(\rho)\left(\fint_{\brhoxo}(1 + |Du|)^{p_2}\log^\gamma(1 + |Du|)\dx\right)^\frac{1}{2}\nonumber\\
		&\leq c(1 + |D\ell|)^{p_2-1}\mathbf{V}^\frac{1}{2}(\rho).
\end{align*}
Taken together, we conclude
\begin{align}\label{logminus}
	(1 + |D\ell|)^{p_2-1}\mathbf{V}^\frac{1}{2}(\rho)&\left(\fint_{\brhoxo}\log^\gamma(1 + |Du|)\dx\right)^\frac{1}{2}
		\leq c(1 + |D\ell|)^{p_2-1}\mathbf{V}^\frac{1}{2}(\rho).
\end{align}

In estimating II we begin by using the VMO condition \vmo, and Young's inequality
\begin{align*}
	\text{II}
		&\leq \fint_{\brhoxo}\mathbf{v}_{\xo}(x,\rho)(1 + |Du|)^{p_2-1}\big[1 + \log(1 + |Du|)\big]\dx\\
		&\leq c\fint_{\brhoxo}\mathbf{v}_{\xo}(x,\rho)(1 + |D\ell| + |Du-D\ell|)^{p_2-1}\big[1 + \log(1 + |Du|)\big]\dx\\
		&\leq c(1 + |D\ell|)^{p_2-1}\fint_{\brhoxo}\bigg[\mathbf{v}_{\xo}(x,\rho) + \mathbf{v}_{\xo}(x,\rho)\bigg|\frac{Du-D\ell}{1 + |D\ell|}\bigg|^{p_2-1}\bigg]\big[1 + \log(1 + |Du|)\big]\dx\\
		&\leq c(1 + |D\ell|)^{p_2-1}\fint_{\brhoxo}\mathbf{v}_{\xo}(x,\rho)\log(1 + |Du|) + \big[\mathbf{v}_{\xo}(x,\rho)\log(1 + |Du|)\big]^{p_2}\dx
		\nl + c(1 + |D\ell|)^{p_2-1}\fint_{\brhoxo}\mathbf{v}_{\xo}(x,\rho) + [\mathbf{v}_{\xo}(x,\rho)]^{p_2} + \bigg|\frac{Du-D\ell}{1 + |D\ell|}\bigg|^{p_2}\dx.
\end{align*}
Now since $0\leq\mathbf{v}\leq 2L$, we have $\mathbf{v}_{\xo}^{p_2}(x,\rho)\leq[2L]^{p_2-1}\mathbf{v}_{\xo}(x,\rho)$. Applying H\"older's inequality (with exponents $(2,2)$) and Young's inequalities (with exponent pairs $(2,2)$ and $(p_2,\frac{p_2}{p_2-1})$), \eqref{logminus} and \vprop$(iv)$, we compute for $2\leq p_2$
\begin{align*}
	\text{II}&\leq c(1 + |D\ell|)^{p_2-1}\fint_{\brhoxo}\mathbf{v}_{\xo}(x,\rho)\Big[1 + \log(1 + |Du|) + \log^{p_2}(1 + |Du|)\Big] + \bigg|V\bigg(\frac{Du-D\ell}{1 + |D\ell|}\bigg)\bigg|^{2}\dx\\
		&\leq c(1 + |D\ell|)^{p_2-1}\Bigg(\mathbf{V}(\rho) + \fint_{\brhoxo}\bigg|V\bigg(\frac{Du-D\ell}{1 + |D\ell|}\bigg)\bigg|^{2}\dx
		\nl + \mathbf{V}^\frac{1}{2}(\rho)\bigg[\bigg(\fint_{\brhoxo}\log^2(1 + |Du|)\dx\bigg)^\frac{1}{2} + \bigg(\fint_{\brhoxo}\log^{2p_2}(1 + |Du|)\dx\bigg)^\frac{1}{2}\bigg]\Bigg)\\
		&\leq c(1 + |D\ell|)^{p_2-1}\Big(\mathbf{V}(\rho) + \mathbf{V}^\frac{1}{2}(\rho) + \Phi(\xo,D\ell,\rho)\Big).
\end{align*}

In the subquadratic case, the calculations on $S_+$ are identical. On $S_-$ we change only the exponent in Young's inequality for a single term (using the pair $(\frac{2}{p_2-1},\frac{2}{3-p_2})$) to compute
\begin{align*}
	(1 + &|D\ell|)^{p_2-1}\fint_{\brhoxo}\mathbf{v}_{\xo}(x,\rho)\bigg|\frac{Du-D\ell}{1 + |D\ell|}\bigg|^{p_2-1}\big[1 + \log(1 + |Du|)\big]\dx\\
		&\leq (1 + |D\ell|)^{p_2-1}\fint_{\brhoxo}\mathbf{v}_{\xo}^{\frac{2}{3-p_2}}(x,\rho)\big[1 + \log(1 + |Du|)\big]^{\frac{2}{3-p_2}} + \bigg|\frac{Du-D\ell}{1 + |D\ell|}\bigg|^{2}\dx\\
		&\leq c(1 + |D\ell|)^{p_2-1}\Bigg(\mathbf{V}(\rho) + \Phi(\xo,D\ell,\rho) + \mathbf{V}^{\frac{1}{2}}(\rho)\bigg(\fint_{\brhoxo}\log^{\frac{4}{3-p_2}}(1 + |Du|^2)\dx\bigg)^{\frac{1}{2}}\Bigg)\\
		&\leq c(1 + |D\ell|)^{p_2-1}\Big(\mathbf{V}(\rho) + \Phi(\xo,D\ell,\rho) + \mathbf{V}^{\frac{1}{2}}(\rho)\Big).
\end{align*}

Owing to \vmocontinu\ we have for III and all $p_2>1$
\begin{align*}
	\text{III}	
		&\leq c\fint_{\brhoxo}\omega_{\xi}\big(|u - \lx|\big)(1 + |Du|)^{p(x) - 1}\dx\\
		&\leq c\left(1 + |D\ell|\right)^{p_2-1}\fint_{\brhoxo}\omega_{\xi}\big(|u - \lx|\big)\bigg[1 + \bigg|\frac{Du-D\ell}{1 + |D\ell|}\bigg|^{p_2-1}\bigg]\dx.
\end{align*}
Now when $p\geq2$ we use Young's inequality (with exponent pair $(p_2,\frac{p_2}{p_2-1})$), the fact $\omega_{\xi}\leq1$, \vprop$(iv)$, then Jensen's inequality to calculate
\begin{align}\label{j0}
	 \fint_{\brhoxo}\omega_{\xi}&\big(|u - \lx|\big)\bigg[1 + \bigg|\frac{Du-D\ell}{1 + |D\ell|}\bigg|^{p_2-1}\bigg]\dx\\
		&\leq c\Bigg[\omega_{\xi}\bigg(\fint_{\brhoxo}\big|u - \lx\big|\dx\bigg) + \fint_{\brhoxo}\bigg|V\bigg(\frac{Du-D\ell}{1 + |D\ell|}\bigg)\bigg|^{2}\dx\Bigg]\nonumber.
\end{align}
When $1<p_2<2$, the estimates on $S_+$ are analogous to the superquadratic case, and on $S_-$ we change only the exponents in Young's inequality to $(\frac{2}{p_2-1},\frac{2}{3-p_2})$, to deduce
\begin{align}\label{j2}
	 \fint_{\brhoxo}\omega_{\xi}&\big(|u - \lx|\big)\bigg[1 + \bigg|\frac{Du-D\ell}{1 + |D\ell|}\bigg|^{p_2-1}\bigg]\chi_{S_-}\dx\\
		&\leq c\Bigg[\omega_{\xi}\bigg(\fint_{\brhoxo}\big|u - \lx\big|\dx\bigg) + \fint_{\brhoxo}\bigg|V\bigg(\frac{Du-D\ell}{1 + |D\ell|}\bigg)\bigg|^{2}\dx\Bigg]\nonumber.
\end{align}
By comparing \eqref{j0} to \eqref{j2}, we see that for any $p_2>1$ we have
\begin{align*}
	\text{III}	&\leq c(1 + |Du|)^{p_2-1}\Bigg[\omega_{\xi}\bigg(\fint_{\brhoxo}\big|u - \lx\big|\dx\bigg) + \fint_{\brhoxo}\bigg|V\bigg(\frac{Du-D\ell}{1 + |D\ell|}\bigg)\bigg|^{2}\dx\Bigg].
\end{align*}
Since the inhomogeneity $b$ satisfies condition \vmocontrol, we can estimate via Corollary~\ref{affinest}~$(i)$
\begin{align*}
	\text{IV}
		\leq L\fint_{\brhoxo}(1 + |Du|)^{p_2 - 1}\rho\dx
		\leq c(1 + |D\ell|)^{p_2 - 1}\rho.
\end{align*}

Assembling our terms, perhaps restricting $\rhoo$ to ensure $\mathbf{V}(\rho)<1$, we have in either case
\begin{align*}
	\bigg|\fint_{\brhoxo}&\big(D_{z}a(\cdot,\urho,D\ell)\big)_{\xo,\rho}Dv\cdot D\varphi\dx\bigg|\\
		&\leq c_{1}(1 + |D\ell|)^{p_2-1}\Bigg(\mu\Big(\sqrt{\Phi(\xo,D\ell,\rho)}\Big)\sqrt{\Phi(\xo,D\ell,\rho)} + \mathbf{V}^\frac{1}{2}(\rho) + \Phi(\xo,D\ell,\rho)
		\nl \qquad\qquad + \omega_{\xi}\bigg(\fint_{\brhoxo}\big|u - \lx\big|\dx\bigg) + \rho\Bigg)\|D\varphi\|_{C(\brhoxo,\RN)}.
\end{align*}
This shows the claim for test functions satisfying $\|D\varphi\|_{L^\infty(\brhoxo,\RN)} = 1$, the full result follows via rescaling of the test function.\end{proof}
\section{Application of the $\mathcal{A}$-harmonic approximation lemma}\label{aharmapp}
We now recall
\begin{align*}
	\Psi(\xo,\ell,\rho)&:=\fint_{\brhoxo}\bigg|V\bigg(\frac{u - \ell}{\rho(1 + |D\ell|)}\bigg)\bigg|^2\dx,\qquad
	\Phi(\xo,D\ell,\rho):=\fint_{\brhoxo}\bigg|V\bigg(\frac{Du - D\ell}{1 + |D\ell|}\bigg)\bigg|^2\dx,\\
\end{align*}
and write
\begin{align*}
	 \Upsilon(\xo,\ell,\rho)&:=\omega_{\xi}\bigg(\fint_{\brhoxo}\big|u-\lx\big|\dx\bigg),\qquad
	 M(\xo,\rho):=\rho\fint_{\brhoxo}|Du|\dx.
\end{align*}
We further define $E(\xo,\ell,\rho):=\Psi(\xo,\ell,\rho) + \mathbf{V}^{\frac{1}{2}}(\rho) + \Upsilon(\xo,\ell,\rho) + \omega_p(\rho) + \rho.$ In view of this notation, we can estimate the Caccioppoli inequality from Lemma~\ref{eqcacc} by
\begin{align}\label{shortcacc}
	\Phi\Big(\xo,D\ell,\frac{\rho}{2}\Big)\leq c_{c}E(\xo,\ell,\rho).
\end{align}
Plugging \eqref{shortcacc} into Lemma~\ref{aharmy} on $\bhalfrhoxo$ we deduce that provided $\Phi\left(\xo,D\ell,\frac{\rho}{2}\right)\leq \frac{1}{16}$, there holds
\begin{align*}
	\bigg|\fint_{\bhalfrhoxo}&\big(D_{z}a(\cdot,\lx,D\ell)\big)_{\xo,\frac{\rho}{2}}Dv\cdot D\varphi\dx\bigg|\\
		&\leq c_{1}(1 + |D\ell|)^{p_2-1}\bigg[\Phi\big(\xo,D\ell,\rho/2\big) + \mu\Big(\sqrt{\Phi(\xo,D\ell,\rho/2)}\Big)\sqrt{\Phi(\xo,D\ell,\rho/2)}
		\nl + \mathbf{V}^{\frac{1}{2}}\big(\rho/2\big) + \Upsilon(\xo,\ell,\rho) + \rho\bigg]\\
		&\leq c_{1}c_{c}(1 + |D\ell|)^{p_2-1}\left(E(\xo,\ell,{\rho}) + \mu\left(\sqrt{E\left(\xo,\ell,{\rho}\right)}\right)\sqrt{E\left(\xo,\ell,\rho\right)}\right)\\
		&\leq c_{a}(1 + |D\ell|)^{p_2-1}\left[\sqrt{E(\xo,\ell,{\rho})} + \mu\left(\sqrt{E\left(\xo,\ell,{\rho}\right)}\right)\right]\sqrt{E\left(\xo,\ell,\rho\right)},
\end{align*}
where we have relabelled the constant.

Having established this preliminary estimate, we can fix these excesses small enough to invoke the $\mathcal{A}$-harmonic approximation lemma. The a priori bounds on the solution to the linearised PDE, combined with our Caccioppoli inequality, allow us to demonstrate a preliminary rescaling estimate on the Campanato style excess functional $\Phi$. This estimate is then iterated, and finally an interpolation argument is provided to reproduce the estimate at all scales. 

\subsection{$\mathcal{A}$-harmonic approximation}

For $\delta_{0}$ to be chosen later we now restrict $\rho$ to be small enough to ensure
\begin{align}\label{smallness}
	\kappa := \sqrt{E\big(\xo,\lrhoxo,\rho\big)}\leq 1
	\qquad\mbox{and}\qquad
	\sqrt{E\big(\xo,\lrhoxo,\rho\big)} + \mu\Big(\sqrt{E\big(\xo,\lrhoxo,\rho\big)}\Big) \leq \delta_{_{0}}.
\end{align}

We now set $\ell = \lhalfrhoxo$ so Lemma~\ref{aharmy} is satisfied with
\begin{align*}
	\mathcal{A} &:=  \frac{\big(D_{z}a(\cdot,\uhalfrho,D\lhalfrhoxo)\big)_{\xo,\frac{\rho}{2}}}{c_{a}(1 + |D\lhalfrhoxo|)^{p_2-2}}
	\qquad\mbox{and}\qquad
	w := \frac{u - \lhalfrhoxo}{1 + |D\lhalfrhoxo|}.
\end{align*}
Given any $\varepsilon>0$, this implies via Lemma~\ref{aharmonic} the existence of an $\mathcal{A}-$harmonic function $h$ satisfying the a priori estimates
\begin{align}\label{pest}
	\sup_{\bhalfrhoxo}\bigg(2\bigg|\frac{h}{\rho}\bigg| + |Dh| + \frac{\rho}{2}|D^2h|\bigg) &\leq c_{h}
\qquad\mbox{and}\qquad
	\fint_{\bhalfrhoxo}\bigg|V\bigg(\frac{w - \kappa h}{\rho/2}\bigg)\bigg|^2\dx \leq \kappa^2\varepsilon,
\end{align}
provided $\delta_{0}$ is small enough.

\section{ Preliminary decay estimate}
Note that via Taylor's theorem  and \eqref{pest} we immediately have
\begin{align}\label{taylor}
	\sup_{\bthetarhoxo}|h(x) - h(\xo) - Dh(\xo)(x-\xo)| \leq \sup_{\brhoxo}|D^2h(\theta\rho)^2|\leq c_{h}\theta^2\rho
\end{align}
for any $x\in\bthetarhoxo$ where $\theta\in(0,\frac{1}{4})$.

We further impose the smallness condition $\kappa\leq c_h^{-1}$ on \eqref{smallness}, which ensures
\begin{align*}
	\kappa\frac{|h(x) - h(\xo) - Dh(\xo)(x-\xo)|}{\theta\rho} \leq c_{_{h}}\kappa\theta<1,
\end{align*}
and so by \vprop$(iv)$ we have
\begin{align}
	\fint_{\bthetarhoxo}\bigg|V\bigg(\kappa\frac{h - h(\xo) - Dh(\xo)(x-\xo)}{\theta\rho}\bigg)\bigg|^2\dx \leq c_{h}^2\theta^2\kappa^2.\label{taylorv}
\end{align}
\begin{lemma}\label{prdees}
For every $M>2$ there exist constants $0< \hat\rho,\theta<\frac{1}{4}$ such that whenever $\Psi(\xo, D\ell, \theta\rho)\leq\frac{1}{16}$ and $|D\ell|<M$ and the smallness conditions
\begin{align}
	\Psi(\xo,\lrhoxo,\rho) < \bigg(\frac{1}{4}\frac{\theta^{n+1}}{n+2}\bigg)^2, \qquad
	E(\xo,\lrhoxo,{\rho})\leq\frac{1}{c^2_{h}},\qquad
	\varepsilon\leq \theta^{n + 2 +\max\{2,p_2\}},\nonumber\\
	\rho < \hat\rho\qquad
	\mbox{and}\qquad\sqrt{E(\xo,\lrhoxo,{\rho})} + \mu\Big(\sqrt{E(\xo,\lrhoxo,\rho)}\Big) \leq \delta_{0}
\label{nonlemmasmallness}
\end{align}
hold, then for all $k\in\N$ we have
\begin{align*}
	\Psi(\xo,\lplusthetarhoxo,\theta\rho) &\leq c_{d}\theta^2E\left(\xo,\lplusrhoxo,\rho\right).
\end{align*}

\end{lemma}
\begin{proof}[Proof of Lemma~\ref{prdees}:]
Using Corollary~\ref{gradscales} and Lemma~\ref{minil} with $\lambda = \frac{1}{\theta\rho(1 + |D\lplushalfrhoxo|)}$ and \vprop$(ii)$ then \eqref{pest} and \eqref{taylorv}, we write $\hat{p}=\max\{2,p_2\}$ and calculate
\begin{align*}
	\Psi(\xo,\lplusthetarhoxo,\theta\rho)
		&\leq 2^{p_2}\fint_{\bthetarhoxo}\bigg|V\bigg(\frac{w-\kappa(h(\xo) + Dh(\xo)(x-\xo)}{\theta\rho}\bigg)\bigg|^{2}\dx\\
		&\leq c(n,p_2)\theta^{-n-\hat{p}}\fint_{\bhalfrhoxo}\bigg|V\bigg(\frac{w-\kappa h}{\rho/2}\bigg)\bigg|^{2}\dx
		\nl+2^{2{p_2}}\fint_{\bthetarhoxo}\bigg|V\bigg(\kappa\frac{h - h(\xo) - Dh(\xo)(x-\xo)}{\theta\rho}\bigg)\bigg|^{2}\dx\\
		&\leq c(n,p_2)\theta^{-n-\hat{p}}\kappa^2\varepsilon + 2^{2{p_2}}c_{_{h}}^2\theta^2\kappa^2.
\end{align*}
Choosing $\varepsilon\leq \theta^{n + 2 +\max\{2,\hat p\}}$ and keeping in mind our definition of $\kappa$, this is simply
\begin{align}\label{almostdecay}
	\Psi(\xo,\lplusthetarhoxo,\theta\rho) &\leq c\theta^2E(\xo,\lplushalfrhoxo,\rho).
\end{align} 

To show $E(\xo,\lplushalfrhoxo,\rho) \leq cE(\xo,\lplusrhoxo,\rho)$, we first consider the term $\Psi$. By \vprop$(i)$~and~$(ii)$, together with Corollary~\ref{gradscales} we have
\begin{align}\label{praline}
	\Psi(\xo,\lplushalfrhoxo,\rho) 
		&\leq \Bigg[\frac{1 + |D\lplusrhoxo|}{1 + |D\lplushalfrhoxo|}\Bigg]^{\hat{p}}\fint_{\brhoxo}\bigg|V\bigg(\frac{u - \lplushalfrhoxo}{\rho(1 + |D\lplusrhoxo|)}\bigg)\bigg|^2\dx\\
		&\leq c\fint_{\brhoxo}\bigg|V\bigg(\frac{u - \lplusrhoxo}{\rho(1 + |D\lplusrhoxo|)}\bigg)\bigg|^2 + \bigg|V\bigg(\frac{\lplushalfrhoxo - \lplusrhoxo}{\rho(1 + |D\lplusrhoxo|)}\bigg)\bigg|^2\dx.\nonumber
\end{align}
Estimate \eqref{dlest1} from Lemma~\ref{affinegrads} and the fact that $\fint_{\bhalfrhoxo}D\ell(x-\xo)\dx=0$ give us the pointwise estimate
\begin{align}\label{utrick}
	\big|\lplushalfrhoxo - \lplusrhoxo\big| &\leq \big|\urho - \uplushalfrho\big| + \rho\big|D\lplushalfrhoxo - D\lplusrhoxo\big|\\
		&\leq\bigg|\fint_{\bhalfrhoxo}u - \lplusrhoxo\dx\bigg| + 2(n+2)\fint_{\bhalfrhoxo}\big|u-\lplusrhoxo\big|\dy\nonumber\\
		&\leq 2^{n+2}(n+3)\fint_{\brhoxo}\big|u-\lplusrhoxo\big|\dy.\nonumber
\end{align}
By the almost-convexity of $V$ given in \eqref{almostconvex} we calculate via Jensen's inequality
\begin{align}\label{fruitjuice}
	\fint_{\brhoxo}\bigg|V\bigg(\frac{\lplushalfrhoxo - \lplusrhoxo}{\rho(1 + |\lplusrhoxo|)}\bigg)\bigg|^2\dx
		&\leq\fint_{\brhoxo}\bigg|V\bigg(\frac{2^{n+2}(n+3)\fint_{\brhoxo}|u(y)-\lplusrhoxo|\dy}{\rho(1 + |\lplusrhoxo|)}\bigg)\bigg|^2\dx\nonumber\\
		&\leq c\fint_{\brhoxo}\bigg|V\bigg(\frac{u-\lplusrhoxo}{\rho(1 + |\lplusrhoxo|)}\bigg)\bigg|^2\dx.
\end{align}
Combining \eqref{praline} and \eqref{fruitjuice}, we have
\begin{align*}
	\Psi(\xo,\lplushalfrhoxo,\rho)
		&\leq c(n,p_2)\Psi(\xo,\lplusrhoxo,\rho),
\qquad\text{provided}\qquad
\Psi(\xo,\lplusrhoxo,\rho)\leq\frac{1}{4}\bigg(\frac{\theta^{n+1}}{n+2}\bigg)^2.
\end{align*}

We now show $\Upsilon(\xo,\lplushalfrhoxo,\rho)\leq c\,\Upsilon(\xo,\lplusrhoxo,\rho)$, which will follow from the concavity (and hence subadditivity) of $\omega_{\xi}$, once we use Lemma~\ref{minil}~$(i)$ to compute as in \eqref{utrick}
\begin{align*}
		\fint_{\brhoxo}|u-\lplushalfrhoxo(\xo)|\dx
		&\leq \fint_{\brhoxo}|u-\urho|\dx + |\uplushalfrho - \urho|\\
		&\leq 2^{n+1}\fint_{\brhoxo}|u-\urho|\dx.
\end{align*}
Since $\omega_{\xi}$ is concave we have
\begin{align*}
	\Upsilon(\xo,\lplushalfrhoxo,\rho)\leq 2^{n+1}\Upsilon(\xo,\lplusrhoxo,\rho).
\end{align*}

Collecting terms, and noting that the other terms in $E\left(\xo,\lplusrhoxo,\rho\right)$ are monotone in $\rho$, we have shown whenever
\begin{align*}
	\Psi(\xo,\lplusrhoxo,\rho)\leq \bigg(\frac{1}{4}\frac{\theta^{n+1}}{n+2}\bigg)^2,
	\qquad \mbox{there holds}\qquad
	E(\xo,\lplushalfrhoxo,\rho)\leq c(n,p_2)E(\xo,\lplusrhoxo,\rho).
\end{align*}
Plugging this into \eqref{almostdecay} we conclude
\begin{align}\label{firstit}
	\Psi(\xo,\lplusthetarhoxo,\theta\rho) &\leq c_{d}\theta^2E(\xo,\lplusrhoxo,\rho),
\end{align}
which is the desired estimate.\end{proof}

\section{ Proof of Theorem~\ref{vmoreg}}

\subsection{Choice of constants}
We now take $\gamma<n$ to be fixed later, and set
\begin{align}\label{choice}\begin{array}{ll}
	\iota = \min\Big\{\frac{1}{6c_{c}}\big(\frac{\theta^{n+1}}{4(n+2)}\big)^{\max\{2,p_2\}}\beta, \big(\frac{1}{4}\frac{\theta^{n+1}}{n+2}\big)^2\Big\},\qquad\qquad
		&\beta = \big(\frac{\theta^{n}}{2c(p)}\big)^{\max\{2,p_2\}},\\
	\theta = \min\Big\{\big(\frac{1}{5c_{_{d}}}\big)^{\frac{1}{2}},\frac{1}{8}, \big(\frac{1}{2}\big)^{\frac{1}{n-\gamma}}\Big\},
		&\hat\rho \leq \min\big\{\frac{\rhoo}{2},{\sigma},\iota\big\},
\end{array}
\end{align}
and ensure $\sigma$ is small enough to satisfy $\omega_{\xi}(\sigma), \mathbf{V}(\sigma), \omega_p(\sigma)\leq\iota$. These constants depend only on $n$, $N$, $M$, $\gamma_1$, $\gamma_2$, $L/\nu$, $E$, $\omega_p,$ and $\mu$.

We can now iterate this procedure to show the following:
\subsection*{Almost BMO estimate}
\addcontentsline{toc}{subsection}{Almost BMO estimate}
\begin{lemma}\label{almostbmo}
For every $M>2$ there exist constants $0< \iota, \beta, \sigma,\hat\rho,\theta<1$ such that whenever the smallness conditions
\begin{align}
	\Psi(\xo,\lrhoxo,\rho) < \iota, \quad\
	\Phi(\xo,&D\lrhoxo,\rho) < \beta, \quad\
	M(\xo,\rho) < \sigma, \quad\
	\text{and}\quad\ \rho < \hat\rho
\end{align}
hold, then for all $k\in\N$ we have
\begin{align*}
	\Psi(\xo,\lthetakrhoxo,\theta^k\rho) < \iota, \qquad
	\Phi(\xo,D\lthetakrhoxo,\theta^k\rho) < \beta, \qquad\mbox{and}\qquad
	M(\xo,\theta^k\rho) < \sigma.
\end{align*}
\end{lemma}
\begin{proof}[Proof of Lemma~\ref{almostbmo}:]
In estimating $M(\xo,\theta^k\rho)$ we assume only that $M(\xo,\rho)<\sigma$ and $\Phi(\xo,D\lthetarhoxo,\theta\rho)<\iota$. By the principle of induction it suffices to show that $M(\xo,\theta^k\rho)<\sigma$ and $\Phi(\xo,D\lthetakrhoxo,\theta^k\rho)<\iota$ imply $M(\xo,\theta^{k+1}\rho)<\sigma$.

We begin by calculating
\begin{align}\label{join}
	M(\xo,\theta^{k+1}\rho) 
		&\leq \theta^{k+1}\rho\fint_{B_{\theta^{k+1}\rho}(\xo)}\big|Du-\dnuk\big|\dx + \theta^{k+1}\rho\big|\dnuk\big|\\
		&\leq \theta^{1-n}\theta^{k}\rho\fint_{B_{\theta^{k}\rho}(\xo)}\big|Du-\dnuk\big|\dx + \theta M(\xo,\theta^k\rho).\nonumber
\end{align}
Again writing $\hat{p} = \max\{2,p_2\}$, we note that Corollary~\ref{affinest}~$(ii)$ together with \vprop$(iv)$ and H\"older's inequality let us calculate
\begin{align*}
	\fint_{B_{\theta^{k}\rho}(\xo)}\frac{|Du-\dnuk|}{1 + |\dnuk|}\dx
		&\leq\frac{1 + |D\lthetakrhoxo|}{1 + |\dnuk|}\fint_{B_{\theta^{k}\rho}(\xo)}\frac{|Du-\dnuk|}{1 + |D\lthetakrhoxo|}\dx\\
		&\leq 4\bigg(\fint_{B_{\theta^{k}\rho}(\xo)}\bigg|V\bigg(\frac{Du-D\lthetakrhoxo}{1 + |D\lthetakrhoxo|}\bigg)\bigg|^2\dx\bigg)^{\frac{1}{\hat{p}}}\\
		&\leq 4c(p_2)\Big(\Phi(\xo,D\lthetakrhoxo,\theta^k\rho)\Big)^{\frac{1}{\hat{p}}}\\
		&=: 4c_m\Big(\Phi(\xo,D\lthetakrhoxo,\theta^k\rho)\Big)^{\frac{1}{\hat{p}}}.
\end{align*}
In order to estimate the first term on the last line of \eqref{join}, we can first compute
\begin{align*}
	\theta^{k}\rho\fint_{B_{\theta^{k}\rho}(\xo)}\big|Du-\dnuk\big|\dx
		 &= \theta^{k}\rho\big(1 + |\dnuk|\big)\fint_{B_{\theta^{k}\rho}(\xo)}\frac{|Du-\dnuk|}{1 + |\dnuk|}\dx\\
		&\leq 4c_m\theta^{k}\rho\Big(\Phi(\xo,D\lthetakrhoxo,\theta^k\rho)\Big)^{\frac{1}{\hat{p}}}\bigg(1 + \fint_{B_{\theta^{k}\rho}(\xo)}|Du|\dx\bigg)\\
		&= 4c_m\Big(\Phi(\xo,D\lthetakrhoxo,\theta^k\rho)\Big)^{\frac{1}{\hat{p}}}M(\xo,\theta^k\rho)
		 + 4c_m\theta^{k}\rho\Big(\Phi(\xo,D\lthetakrhoxo,\theta^k\rho)\Big)^{\frac{1}{\hat{p}}}.
\end{align*}
Hence, via \eqref{join} we see
\begin{align*}
	M(\xo,\theta^{k+1}\rho)
		&\leq \theta^{1-n}\theta^{k}\rho\fint_{B_{\theta^{k}\rho}(\xo)}\big|Du-\dnuk\big|\dx + \theta M(\xo,\theta^k\rho)\\
		&\leq 4_m\theta^{1-n}\Big(\Phi(\xo,D\lthetakrhoxo,\theta^k\rho)\Big)^{\frac{1}{\hat{p}}}M(\xo,\theta^k\rho)
		\nl + 4c_m\theta^{1-n}\theta^{k}\rho\Big(\Phi(\xo,D\lthetakrhoxo,\theta^k\rho)\Big)^{\frac{1}{\hat{p}}}
		 + \theta M(\xo,\theta^k\rho)\\
		&\leq \frac{M(\xo,\theta^k\rho)}{4} + \frac{\theta^{k}\rho}{4} + \frac{M(\xo,\theta^k\rho)}{8}
		\leq \sigma,
\end{align*}
whenever
\begin{align*}
	\Phi(\xo,D\lthetakrhoxo,\theta^k\rho) \leq \bigg(\frac{\theta^{n}}{2c_m}\bigg)^{\hat{p}},\mathand \rho\leq \frac{\sigma}{\theta^k},
\end{align*}
which holds by \eqref{choice}. Since $\hat\rho\leq\sigma$ satisfies
\begin{align*}
	\mathbf{V}(\hat\rho)<\iota^2,\qquad
	\omega_p(\hat\rho)<\iota,\qquad
	\hat\rho<\iota,
\qquad\text{and}\qquad 
	\omega_{\xi}(\sigma)\leq\iota,
\end{align*}
we can use estimate \eqref{firstit} with $\theta^k\rho$ in place of $\rho$ to establish
\begin{align*}
	\Psi(\xo,\lplusthetakonerhoxo,\theta^{k+1}\rho) &\leq c_{d}\theta^2E(\xo,\lplusthetakrhoxo,\theta^k\rho)
		\leq 5c_{d}\theta^2\iota
		\leq\iota,
\end{align*}
provided $\theta\leq\big(\frac{1}{6c_{d}}\big)^{\frac{1}{2}}$, which holds by \eqref{choice}. Finally we show the estimate for $\Phi$. Note first of all that via Remark~\ref{caccoro} we can calculate
\begin{align}\label{jill}
	\Phi(\xo,D\lplusthetakrhoxo,\theta^{k+1}\rho) &\leq c_{c}E(\xo,\lplusthetakrhoxo,\theta^k\rho)
		\leq 6c_{c}\iota,
\end{align}
so it remains to estimate $\Phi(\xo,D\lplusthetakonerhoxo,\theta^{k+1}\rho)$ in terms of $\Phi(\xo,D\lplusthetakrhoxo,\theta^{k+1}\rho)$. We begin by noting \vprop$(i)$ and $(ii)$ with Corollary~\ref{gradscales} together yield
\begin{align}\label{phiphi}
	\Phi(\xo,D\lplusthetakonerhoxo,\theta^{k+1}\rho) 
		&\leq \fint_{B_{\theta^{k+1}\rho}(\xo)}\bigg|V\bigg(\frac{1 + |D\lplusthetakrhoxo|}{1 + |D\lplusthetakonerhoxo|}\frac{Du - D\lplusthetakonerhoxo}{1 + |D\lplusthetakrhoxo|}\bigg)\bigg|^2\dx\\
		&\leq 2^{\hat{p}}\fint_{B_{\theta^{k+1}\rho}(\xo)}\bigg|V\bigg(\frac{Du - D\lplusthetakonerhoxo}{1 + |D\lplusthetakrhoxo|}\bigg)\bigg|^2\dx\nonumber\\
		&\leq 2^{p_2+\hat{p}}\Bigg[\Phi(\xo,D\lplusthetakrhoxo,\theta^{k+1}\rho)\nonumber
			 + \bigg|V\bigg(\frac{D\lplusthetakrhoxo - D\lplusthetakonerhoxo}{1 + |D\lplusthetakrhoxo|}\bigg)\bigg|^2\Bigg].
\end{align}
When considering the second term, we begin by noting via Lemma~\ref{affinegrads}
\begin{align*}
	\big|D\lplusthetakrhoxo - D\lplusthetakonerhoxo\big| &\leq \bigg(\frac{n+2}{\theta^{k+1}\rho}\bigg)\fint_{B_{\theta^{k+1}\rho}(\xo)}\big|u-\lplusthetakrhoxo\big|\dx\\
		&\leq \bigg(\frac{n+2}{\theta^{n+1}}\bigg)\fint_{B_{\theta^{k}\rho}(\xo)}\bigg|\frac{u-\lplusthetakrhoxo}{\theta^k\rho}\bigg|\dx.
\end{align*}
So \vprop$(i)$, the almost-convexity of $V$ as per \eqref{almostconvex}, and Jensen's inequality imply
\begin{align*}
	\bigg|V\bigg(\frac{D\lplusthetakrhoxo - D\lplusthetakonerhoxo}{1 + |D\lplusthetakrhoxo|}\bigg)\bigg|^2
		&\leq \bigg|V\bigg(\frac{n+2}{\theta^{n+1}}\fint_{B_{\theta^{k}\rho}(\xo)}\bigg|\frac{u-\lplusthetakrhoxo}{\theta^k\rho(1 + |D\lplusthetakrhoxo|)}\bigg|\dx\bigg)\bigg|^2\\
		&\leq \bigg(\frac{n+2}{\theta^{n+1}}\bigg)^{\hat{p}} \fint_{B_{\theta^{k}\rho}(\xo)}\bigg|V\bigg(\bigg|\frac{u-\lplusthetakrhoxo}{\theta^k\rho(1 + |D\lplusthetakrhoxo|)}\bigg|\bigg)\bigg|^2\dx.
\end{align*}
Plugging this into \eqref{phiphi}, in view of \eqref{jill} and 
	$\iota\leq \frac{1}{6c_{c}}\bigg(\frac{\theta^{n+1}}{4(n+2)}\bigg)^{\hat{p}}\beta$
from \eqref{choice}, we find
\begin{align*}
	\Phi(\xo,D\lplusthetakonerhoxo,\theta^{k+1}\rho)
		&\leq 2^{p_2+\hat{p}}\Phi(\xo,D\lplusthetakrhoxo,\theta^{k+1}\rho)l
			 + 2^{p_2}\bigg(2\frac{n+2}{\theta^{n+1}}\bigg)^{\hat{p}}\Phi(\xo,D\lplusthetakrhoxo,\theta^{k}\rho)\\
		&\leq 2^{p_2+\hat{p}}\Bigg[6c_{c} + \bigg(\frac{n+2}{\theta^{n+1}}\bigg)^{\hat{p}}\Bigg]\iota
		< 6c_{c}\left(4\frac{n+2}{\theta^{n+1}}\right)^{\hat{p}}\iota
		\leq\beta.
\end{align*}
\end{proof}

\subsection*{Iteration}
We proceed to calculate via Corollary~\ref{affinest}~$(ii)$, \vprop$(iv)$ and H\"older's inequality, writing $q = \min\{\frac{1}{2},\frac{1}{p_2}\}$ and keeping in mind our choice of $\beta$ from \eqref{choice}
\begin{align*}
	\int_{\bthetakone}\big|Du\big|\dx &\leq \int_{\bthetakone}\big|Du-\dnuk\big|\dx + \alpha_n\big(\theta^{k+1}\rho\big)\big|\dnuk\big|\\
		&\leq 2\big(1 + |\dnuk|\big)\int_{\bthetak}\frac{|Du-\dnuk|}{1 + |\lthetakrhoxo|}\dx
		\nl + \alpha_n\big(\theta^{k+1}\rho\big)^n\big|\dnuk\big|\\
		&\leq 2\alpha_n\big(\theta^{k}\rho\big)^n\Phi^{\frac{1}{\hat{p}}}(\xo,D\lthetarhoxo,\rho) + \alpha_n\left(\theta^{k}\rho\right)^n\big[2\Phi^{\frac{1}{\hat{p}}}(\xo,D\lthetarhoxo,\rho) + \theta^n\big]\big|\dnuk\big| \\
		&\leq c\big(\theta^{k}\rho\big)^n\beta + 2\theta^n\int_{\bthetak}\big|Du\big|\dx\\
		&\leq \theta^\gamma\int_{\bthetak}\big|Du\big|\dx + c\beta\big(\theta^{k}\rho\big)^n,
\end{align*}
for each $k\in\N$ and $\gamma\in(n-1,n)$ by choice of $\theta$ in \eqref{choice}. Setting $f(t) = \int_{B_{t}(\xo)}|Du|\dx$ in Lemma~\ref{iteration}, we deduce that for every $r\in(0,\rho)$
\begin{align}
	\int_{\brxo}\big|Du\big|\dx  \label{morrey}
		\leq c\Bigg[ r^\gamma + \bigg(\frac{r}{\rho}\bigg)^\gamma\int_{\brhoxo}\big|Du\big|\dx\Bigg]
		\leq c\Bigg[1 + \frac{1}{\rho^\gamma}\int_{\Omega}\big|Du\big|\dx\Bigg] r^\gamma,
\end{align}
for some constant depending on $\theta,\gamma,n$ and so ultimately on all of the structure data.

\subsection*{Interpolation}
In order conclude Theorem~\ref{vmoreg}, we require estimates on our renormalised first-order excess functional in terms of the quantities appearing in the characterisation of our singular sets.

We note via \vprop$(iv)$ and H\"older's inequality
\begin{align}
	\Phi(\xo,D\ell_{\xo,\rho},\rho)
		&\leq\fint_{\brhoxo}\bigg|\frac{Du - D\ell_{\xo,\rho}}{1 + |D\ell_{\xo,\rho}|}\bigg|^{p_2}\dx + \bigg(\fint_{\brhoxo}\bigg|\frac{Du - D\ell_{\xo,\rho}}{1 + |D\ell_{\xo,\rho}|}\bigg|^{p_2}\dx\bigg)^{\min\{\frac{2}{\hat{p}},1\}}\nonumber\\
		&\leq 2\bigg(\fint_{\brhoxo}\bigg|\frac{Du - D\ell_{\xo,\rho}}{1 + |D\ell_{\xo,\rho}|}\bigg|^{p_2}\dx\bigg)^{\min\{\frac{2}{\hat{p}},1\}}\nonumber\\
		&\leq c(p_2)\bigg(\fint_{\brhoxo}\bigg|\frac{Du - D\ell_{\xo,\rho}}{1 + |\dnu|}\bigg|^{p_2}\dx+ \bigg|\frac{\dnu - D\ell_{\xo,\rho}}{1 + |\dnu|}\bigg|^{p_2}\bigg)^{\min\{\frac{2}{\hat{p}},1\}}.\label{dweeb}
\end{align}
 Using \eqref{dlest2} from Corollary~\ref{affinegrads} and Poincar\'e's inequality we continue to calculate
\begin{align}
	\bigg|\frac{\dnu - D\ell_{\xo,\rho}}{1 + |\dnu|}\bigg|^{p_2}
		&\leq c(n,p_2)\fint_{\brhoxo}\bigg|\frac{u - \urho - \dnu(x-\xo)}{\rho(1 + |\dnu|)}\bigg|^{p_2}\dx\nonumber\\
		&\leq c(n,p_2)\fint_{\brhoxo}\bigg|\frac{Du - \dnu}{1 + |\dnu|}\bigg|^{p_2}\dx\label{shark}.
\end{align}
We now use the interpolation estimate Lemma~\ref{logconvex}, with $s = p_2$, $p = 1$ and $q = p_2(1 + \frac{\delta}{4})$, we find $\theta = \frac{\delta}{p_2(4+\delta) - 4}\in(0,1)$ to compute
\begin{align*}
	\bigg\|\frac{Du - \dnu}{1 + |\dnu|}\bigg\|_{L^{p_2}} &\leq \bigg\|\frac{Du - \dnu}{1 + |\dnu|}\bigg\|_{L^1}^\theta\bigg\|\frac{Du - \dnu}{1 + |\dnu|}\bigg\|_{L^{p_2(1 + \frac{\delta}{4})}}^{1-\theta}.
\end{align*}
After averaging, this is just
\begin{align}
	\fint_{\brhoxo}\bigg|\frac{Du - \dnu}{1 + |\dnu|}\bigg|^{p_2}\dx
		&\leq \bigg(\fint_{\brhoxo}\bigg|\frac{Du - \dnu}{1 + |\dnu|}\bigg|\dx\bigg)^{\frac{p_2\delta}{p_2(4 + \delta) - 4}}\nonumber\\
		&\qquad \times
		 \bigg(\fint_{\brhoxo}\bigg|\frac{Du - \dnu}{1 + |\dnu|}\bigg|^{p_2(1 + \frac{\delta}{4})}\dx\bigg)^{\frac{4p_2 - 4}{p_2(4 + \delta) - 4}}\nonumber\\
		&\leq \kappa^{\frac{p_2\delta}{p_2(4 + \delta) - 4}} \bigg(\fint_{\brhoxo}\bigg|\frac{Du - \dnu}{1 + |\dnu|}\bigg|^{p_2(1 + \frac{\delta}{4})}\dx\bigg)^{\frac{4p_2 - 4}{p_2(4 + \delta) - 4}}\label{squid},
\end{align}
since $\frac{p_2\delta}{p_2(4 + \delta) - 4} + \frac{4p_2 - 4}{p_2(4 + \delta) - 4} = 1.$ Now writing $\lambda = \big(1 + |\dnu|\big)^{-1}$, we have via Lemma~\ref{frozenexponent} (with $p_0 = p_2(1 + \frac{\delta}{4})$ and $p=p_2$), and Corollary~\ref{affinest}~$(i)$ and $(ii)$
\begin{align*}
	\fint_{\brhoxo}\bigg|\frac{Du - \dnu}{1 + |\dnu|}\bigg|^{p_2(1 + \frac{\delta}{4})}\dx
		&\leq c(p_2)\fint_{\brhoxo} \big|\lambda Du\big|^{p_2(1 + \frac{\delta}{4})} + \big|\lambda\dnu\big|^{p_2(1 + \frac{\delta}{4})}\dx\\
		&\leq c(p_2) + c\lambda^{p_2(1 + \frac{\delta}{4})}\bigg(\fint_{\btworhoxo}1 + |Du|^{p_2}\dx\bigg)^{1 + \frac{\delta}{4}}\\
		&\leq c(p_2) + c|\lambda|^{p_2(1 + \frac{\delta}{4})}(1 + |\dtwo|)^{p_2(1 + \frac{\delta}{4})}\\
		&\leq c(p_2) + c|\lambda|^{p_2(1 + \frac{\delta}{4})}(1 + |\dnu|)^{p_2(1 + \frac{\delta}{4})}\\
		&\leq c=:c_v,
\end{align*}
by definition of $\lambda$. Here, the constant retains the dependencies of the constant from Lemma~\ref{frozenexponent}. Plugging this into \eqref{squid}, we have
\begin{align}
	\fint_{\brhoxo}\bigg|\frac{Du - \dnu}{1 + |\dnu|}\bigg|^{p_2}\dx
		&\leq \bigg(\fint_{\brhoxo}\bigg|\frac{Du - \dnu}{1 + |\dnu|}\bigg|\dx\bigg)^{\frac{p_2\delta}{p_2(4 + \delta) - 4}}\nonumber\\
		&\qquad \times \bigg(\fint_{\brhoxo}\bigg|\frac{Du - \dnu}{1 + |\dnu|}\bigg|^{p_2(1 + \frac{\delta}{4})}\dx\bigg)^{\frac{4p_2 - 4}{p_2(4 + \delta) - 4}}\nonumber\\
		&\leq \kappa^{\frac{p_2\delta}{p_2(4 + \delta) - 4}} c_v^{\frac{4p_2 - 4}{p_2(4 + \delta) - 4}}\label{fish}.
\end{align}

\subsection{Partial regularity}
We note the dependence of $\kappa$ on $\gamma$ (via $\beta$'s dependence on $\theta$), and note that $\kappa\rightarrow0$ as $\gamma\rightarrow n$. Furthermore, since $\rho\fint_{\brhoxo}|Du|\dx<\sigma$, Poincar\'e's inequality implies
\begin{align*}
	\Upsilon(\xo,\ell_{\xo,\rho},\rho) &= \omega_{\xi}\left(\fint_{\brhoxo}|u-\urho|\dx\right)
		\leq \omega_{\xi}\left(\rho\fint_{\brhoxo}|Du|\dx\right)
		\leq \omega_{\xi}\left(\sigma\right)
		\leq \iota.
\end{align*}

Taking some point $\xo\in\Omega$ and $\rho\leq\hat\rho$ satisfying $|D\ell|\leq M$ for fixed $M<\infty$ with
\begin{align}\label{epreg}
	\liminf_{\rho\downarrow0}\fint_{\brhoxo}\big|{Du - \dnu}\bigg|\dx<\kappa,\qquad\text{and}\qquad \liminf_{\rho\downarrow0}\rho\fint_{\brhoxo}|Du|\dx<\sigma,
\end{align}
we can find a $\rho<\hat\rho$ such that the conditions of Lemma~\ref{almostbmo} hold, with $\rho$ at this stage depending on all of the structure conditions. Furthermore, if the conditions of Lemma~\ref{almostbmo} hold at this point $\xo$ and fixed $\rho<\hat\rho$, then by the absolute continuity of the Lebesgue integral, there exists an $R<\rho$ such that these same conditions hold for each $x\in B_{R}(\xo)$. Consequently, we deduce that \eqref{morrey} and \eqref{fish} hold for every $r\leq\frac{R}{4}$ and $y\in B_\frac{R}{4}(\xo)$. This implies $Du$ belongs to the Morrey space $L^{1,\gamma}(B_{\frac{R}{4}}(y),\RnN)$ for $\gamma\in(n-1,n)$, and so the Morrey-Campanato embedding theorem implies $u\in C^{0,\tau}(B_{\frac{R}{4}}(y),\RN)$ for $\tau = 1-n-\gamma$. Note $\tau$ can be chosen to be any value in $(0,1)$ provided $\kappa$ and hence $\sigma$ are chosen accordingly as functions of $\tau$, decaying to $0$ as $\alpha \uparrow 1$. This in turn restricts the neighbourhood $B_{\frac{R}{4}}(\xo)$ on which the estimate holds, since $\hat\rho$ also has dependence on $\gamma$ via $\alpha$ through $\beta$, and so ultimately $\theta$. Indeed, if we take $\kappa=\sigma=0$ we can obtain the conclusion of Theorem~\ref{vmoreg} for every $\tau\in(0,1)$.

By definition we have the inclusions $\Sigma^\tau_{1,u}\subset\Sigma^0_{1,u}$ and $\Sigma^\tau_{2,u}\subset\Sigma^0_{2,u}$. Since these sets are closed, we have $\reg{u}$ and $\reg^\tau(u)$ are relatively open in $\Omega$ and hence open. Furthermore, by Lebesgue's differentiation theorem we have $|\Sigma^0_{1,u}|=0$, and Lemma~\ref{dimred} implies $\dimh(\Sigma^0_{2,u})\leq n-1$. We conclude that $| \Omega\setminus\reg(u)|=0$ and hence $|\Omega\setminus\reg^\tau(u)|=0.\hfill\blacksquare$

\end{document}